\documentclass[a4paper,11pt]{amsart}

\usepackage{amssymb,amsmath,amsthm,amscd}
\usepackage[ansinew]{inputenc} %per le accentate
\usepackage[english]{babel}

\usepackage{newtxtext}
\usepackage{newtxmath}
\usepackage[T1]{fontenc}
\usepackage{mathrsfs}
\usepackage{amscd}
\usepackage{currfile}
\usepackage{hyperref}
\usepackage{bm}
\usepackage{currfile}

\newtheorem{theorem}{Theorem}%[section]
\newtheorem*{theorem*}{Theorem}
\newtheorem{definition}[theorem]{Definition}
\newtheorem{lemma}[theorem]{Lemma}
\newtheorem{proposition}[theorem]{Proposition}
\newtheorem{corollary}[theorem]{Corollary}
\newtheorem{example}[theorem]{Example}
\newtheorem{examples}[theorem]{Examples}
\theoremstyle{definition}
\newtheorem{remark}[theorem]{Remark}
\newtheorem{remarks}[theorem]{Remarks}

\newcommand{\hh}{{\mathbb{H}}}
\newcommand{\HH}{{\mathbb{H}}}

\newcommand{\cc}{{\mathbb{C}}}
\newcommand{\rr}{{\mathbb{R}}}
\newcommand{\zz}{{\mathbb{Z}}}
\newcommand{\nn}{{\mathbb{N}}}

\newcommand{\s}{{\mathbb{S}}}

\newcommand{\SR}{\mathcal{SR}}
\newcommand{\SF}{\mathcal{S}}
\newcommand{\AM}{\mathcal{AM}}

\newcommand{\SC}{\mathcal{SC}}

\newcommand{\I}{\mathcal{I}}
\newcommand{\LL}{\mathcal{L}}
\newcommand{\Pn}{\mathcal{P}}
\newcommand{\dcf}{\dibar}
\newcommand{\dibar}{\overline\partial}
\newcommand{\cdcf}{\partial}

\newcommand\IM{\operatorname{Im}}

\newcommand\sd[1]{{#1}'_s}

\newcommand{\punto}{\bullet}

\newcommand\im{\operatorname{Im}}

\newcommand\Ker{\operatorname{Ker}}
\newcommand\sinc{\operatorname{sinc}}

\newcommand{\ui}{\imath}

\newcommand{\OO}{\Omega}

\newcommand{\mbb}{\mathbb}

\newcommand{\R}{\mbb{R}}

\newcommand{\cdf}{\frac{\partial f}{\partial x}}
\newcommand{\dd[3]}{\frac{\partial #2}{\partial #3}}
\newcommand{\C}{\mbb{C}}

\newcommand{\Zt}{\widetilde{\mathcal Z}}

\newcommand{\Dl}{\mathcal D_\lambda}
\newcommand{\Al}{\mathcal A_\lambda}
\newcommand{\Sl}{\mathcal S_\lambda}
\newcommand{\Dm}{\mathcal D_\mu}
\newcommand{\DL}[1]{\mathcal D_{\lambda_{#1}}}
\newcommand{\LLL}[1]{\mathcal L_{\lambda_{#1}}}
\newcommand{\El}{\mathcal E_\lambda}

\newcommand{\ID}{\widetilde\Delta}

\begin{document}

\title{Eigenvalue problems for slice functions}

\author{Rolf S\"oren Krau{\ss}har}
\address{Fachbereich Fachwissenschaft Mathematik,
Erziehungswissenschaftliche Fakult\"at,
Universit\"at Erfurt,
Nordh\"auser Str. 63, 
99089 Erfurt, Germany
}
\email{soeren.krausshar@uni-erfurt.de}

\author{Alessandro Perotti}
\address{Department of Mathematics, University of Trento, Via Sommarive 14, Trento, Italy}
\email{alessandro.perotti@unitn.it}

%\date{  }
\thanks{The second author was supported by GNSAGA of INdAM, and by the grants ``Progetto di Ricerca INdAM, Teoria delle funzioni ipercomplesse e applicazioni'', and PRIN ``Real and Complex Manifolds: Topology, Geometry and holomorphic dynamics'' of Ministero dell'universit\`a e della ricerca.}

\begin{abstract} This paper addresses particular eigenvalue problems within the context of two quaternionic function theories.  
More precisely, we study two concrete classes of quaternionic eigenvalue problems, the first one for the slice derivative operator in the class of quaternionic slice-regular functions and the second one for the Cauchy-Riemann-Fueter operator in the class of axially monogenic functions. The two problems are related to each other by the four-dimensional Laplace operator and Fueter's Theorem. As an application of a particular case of second order eigenvalue problems, we obtain a representation of axially monogenic solutions for time-harmonic Helmholtz and stationary Klein-Gordon equations.   
\end{abstract}


\maketitle

% \begin{center}
% {\large\texttt{(\currfilename)}}
% \end{center}

\section{Introduction}
Complex function theory represents a powerful toolkit to study eigenvalue problems related to important differential operators arising in harmonic analysis and mathematical physics in two dimensions. Therefore, a strong motivation in mathematical analysis consists in developing higher dimensional analogues in order to tackle similarly corresponding spatial problems. The smallest division algebra that encompasses the three-dimensional space $\mathbb{R}^3$ is the four-dimensional Hamiltonian skew field $\mathbb{H}$ which is not commutative anymore but still associative. It is the first algebra created beyond the complex numbers by applying the well-known Cayley-Dickson duplication process.    
The concept of holomorphic functions however can be generalized in a number of rather different ways to higher dimensional algebras, even in the simplest context of $\mathbb{H}$. Following the Riemann approach one can consider quaternion-valued functions $f$ that are annihilated either from the left or from the right by the linear Cauchy-Riemann-Fueter operator $\dcf := \frac{1}{2}\Big(\frac{\partial }{\partial x_0} + i \frac{\partial }{\partial x_1}  + j \frac{\partial }{\partial x_2}  + k \frac{\partial }{\partial x_3} \Big)$. Here, $i,j,k$ are the quaternionic imaginary units, satisfying $i^2=j^2=k^2=-1$ and $ij=k$, $jk=i$, $ki=j$ as well as $ji=-k$, $kj=-i$, $ik=-j$. 
Functions in the kernel of this operator are nowadays often called left (resp.\ right) monogenic functions or left (right) Fueter regular functions and have been studied by a constantly growing community for more than a century. As a classical reference we recommend for example \cite{bds}. This function theory is often called Clifford analysis. The Cauchy-Riemann-Fueter operator is nothing else than the Euclidean Dirac operator in $\mathbb{R}^4$ associated to the Euclidean flat metric and is a first order square root of the Laplacian $\Delta_4$.    

An alternative function theory, which actually is even more closely related to classical complex-analytic functions, is the theory of quaternionic slice-regular functions which was basically introduced in 2006-2007 by Gentili and Struppa \cite{GeSt2006CR,GeSt2007Adv}. This function theory exploits a particular slice-structure of $\mathbb{H}$ which is explained together with its most important definitions and relevant function classes of the so-called slice functions and slice-regular functions in Section~2 including the basic references.  
\par\medskip\par

Now the aim of our paper is to investigate some eigenvalue problems for quaternionic slice functions, with particular emphasis on slice-regular functions and on axially monogenic functions. 

\par\medskip\par

The right-linear operator that we first consider is the so-called slice derivative operator $\dd{}x$ acting on the classes of slice or slice-regular functions defined in Section~2. We study eigenvalue problems of the form $\dd{f}{x}=f\lambda$ on an axially symmetric domain $\OO$ of the quaternionic space $\hh$, with $\lambda\in\hh$. We show in Proposition \ref{pro:sol_hom} that in the class of slice functions, the solutions to this problem can be written as slice products of anti-slice-regular functions and a quaternionic exponential function. Slice-regular solutions are obtained when the first factor is a slice-constant function on $\OO$. 

Let $\Dl$ be the right-linear operator defined by $\Dl f=\dd{f}{x}-f\lambda$. In Propositions \ref{pro:polynomial}, \ref{pro:lmII} and $\ref{pro:El}$ we study the associated non-homogeneous eigenvalue problem $\Dl{f}=h$ for a polynomial right-hand side $h$ and for a wide class of entire slice-regular functions. These results in particular permit us to solve (Corollary \ref{cor:order2}) second-order eigenvalue problems $\Dl\Dm f=0$ for any choice of quaternionic eigenvalues $\lambda,\mu$. In Corollary \ref{cor:order_m} we extend this result to the $m$-th order eigenvalue problem $\DL{1}\cdots\DL{m} f=0$ by means of new \emph{generalized exponential functions} (see Definition \ref{def:E}) associated with any ordered $k$-tuple $(\mu_1,\ldots,\mu_k)\in\hh^k$.

The second right-linear operator that we consider is the (conjugated) Cauchy-Riemann-Fueter operator $\cdcf$ acting on the class of quaternionic \emph{monogenic} (or \emph{Fueter-regular}) slice functions, i.e., belonging to the kernel of the Cauchy-Riemann-Fueter operator $\dcf$. Using Fueter's Theorem (see e.g.\ \cite{TheInverseFueter}) as a bridge between the two function theories, we are able to apply the results described above to eigenvalue problems for axially monogenic functions. The crucial fact is the relation $\cdcf\circ\Delta_4=\Delta_4\circ\dd{}{x}$  on the space of slice-regular functions on $\OO$, where $\Delta_4$ is the Laplacian of $\rr^4\simeq\hh$. Applying the operator $\Delta_4$ to the slice-regular solutions obtained above, we obtain in Proposition \ref{pro:am} the general axially monogenic solution of the eigenvalue problem $\cdcf f=f\lambda$, with $\lambda\in\hh$, and more generally (Proposition \ref{pro:am_m}) of the $m$-th order eigenvalue problem $\LLL{1}\cdots\LLL{m} f=0$, where $\LL_{\lambda}$ is defined by $\LL_{\lambda}f=\cdcf f-f\lambda$. 
We refer the reader at the beginning of Sect.~4 to references on the study of similar eigenvalue problems in quaternionic and Clifford analysis.

The final section is devoted to applications. We relate the solutions to $\LLL{1}\LLL{2} f = 0$ to axially monogenic solutions of the three-dimensional time-harmonic Helmholtz and stationary massless Klein-Gordon equation on an axially symmetric domain (Proposition \ref{pro:KGH}). 

Let $\Delta_3$ denote the Laplacian operator on $\rr^3\simeq\IM(\hh)$. 
Suppose that $\lambda_1 = I \lambda$ and $\lambda_2 = - I \lambda$, where $I$ is a quaternionic imaginary unit and where $\lambda$ is an arbitrary non-zero real number. Then the solutions to  $\LL_{I\lambda}\LL_{-I\lambda} f = 0$ are axially monogenic solutions to the massless stationary Klein-Gordon equation $(\Delta_3 - \lambda^2) f = 0$ on $\Omega^{*} := \Omega \cap \mathbb{R}^3$. 
If we take $\lambda_1 = \lambda$ and $\lambda_2 =- \lambda$ where again $\lambda$ is supposed to be a non-zero real value, then the solutions to  $\LL_{\lambda}\LL_{-\lambda} f = 0$  are axially monogenic solutions to the time-harmonic Helmholtz equation $(\Delta_3 + \lambda^2) f = 0$ on $\Omega^{*}$.
Finally, we establish an application (Remark \ref{rem:Yukawa}) to the non-homogeneous equation associated to the Klein-Gordon equation, i.e., the Yukawa equation $(\Delta_3 -\lambda^2)f=h$, with $\lambda$ real and $h$ axially monogenic. 
\par\medskip\par
The paper is structured as follows: In Sect.~2 we recall the basic notions of slice function theory on $\hh$. Then, in Sect.~3, we present the eigenvalue problems for the slice derivative operator on slice and on slice-regular functions. Here we also introduce the generalized exponential functions $E_\Lambda$. This represents another essential novelty of this paper, and we give some examples to illustrate the method of solution. In Sect.~4 we recall Fueter's Theorem and present some new results about the Laplacian of a slice-regular function. Then we %introduce a family of axially monogenic polynomials related to the power functions $x^n$. This leads to 
prove the commutativity relation linking $\Delta$, $\cdcf$ and $\dd{}{x}$ and obtain the axially monogenic solutions to the eigenvalue problems for $\cdcf$, in terms of the generalized $\Delta$-exponential functions $E^\Delta_{\Lambda}$. Finally, in Sect.~5, we present applications of the results of Sect.~4 to the Helmholtz, Klein-Gordon and Yukawa equations and round off our paper by presenting some explicit examples.

\section{Preliminaries}\label{sec:pre}

The theory of quaternionic slice-regular functions was introduced in 2006-2007 by Gentili and Struppa \cite{GeSt2006CR,GeSt2007Adv}. We refer the reader for instance to \cite{CSS2009,GeStoSt2013,GhPe_AIM,AlgebraSliceFunctions,DivisionAlgebras} and also to the references therein for precise definitions and for more results on this class of functions. 
Slice function theory is based on the ``slice'' decomposition of the quaternionic space $\hh$. For each imaginary unit $J$ in the sphere
 \[\s=\{J\in\HH\ |\ J^2=-1\}=\{x_1i+x_2j+x_3k\in\HH\ |\ x_1^2+x_2^2+x_3^2=1\},\]
we denote by $\C_J=\langle 1,J\rangle\simeq\C$ the subalgebra generated by $J$. Then it holds
\[\HH=\bigcup_{J\in \s}\C_J, \quad\text{with $\C_J\cap\C_K=\R$\ for every $J,K\in\s,\ J\ne\pm K$.}\]
A (real) differentiable function $f:\OO\subseteq\HH\rightarrow\HH$  is called \emph{(left) slice-regular} \cite{GeSt2007Adv} on the open set $\OO$ if, for each $J\in\s$, the restriction
$f_{\,|\OO\cap\C_J}\, : \, \OO\cap\C_J\rightarrow \HH$
is holomorphic with respect to the complex structure defined by left multiplication by $J$. 

\subsection*{Slice functions}

A different approach to slice regularity was introduced in \cite{GhPe_Trends,GhPe_AIM}, making use of the concept of \emph{slice function}. These are exactly the quaternionic functions that are compatible with the slice structure of $\hh$. %We recall their definition.
Given a set $D\subseteq\C$ that is invariant with respect to complex conjugation, a function $F: D\to \hh\otimes\C$ that satisfies $F(\overline z)=\overline{F(z)}$ for every $z\in D$ is called a \emph{stem function} on $D$. 
Here the conjugation in $\hh\otimes\C$ is the one induced by complex conjugation of the second factor.

Let $D$ be open and let $\OO_D=\cup_{J\in\s_\hh}\Phi_J(D)\subset\hh$, where for any $J\in\s$, the map $\Phi_J:\C\to\C_J$ is the canonical isomorphism defined by $\Phi_J(a+i b):=a+Jb$.
Open domains in $\hh$ of the form $\OO=\OO_D$ are called \emph{axially symmetric} domains. 
Note that every axially symmetric domain $\OO$ is either a \emph{slice domain} if $\OO\cap \R\ne\emptyset$, or it is a \emph{product domain}, namely if $\OO\cap \R=\emptyset$. Moreover, any axially symmetric open set can be represented as a union of a family of domains of these two types. %The product domains have this name because they are homeomorphic to a product $D^+\times \s$, with $D^+$ open domain in the complex upper half-plane.

The stem function $F=F_1+iF_2$  on $D$ (with $F_1$ and $F_2$ $\hh$-valued functions on $D$) induces the \emph{slice function} $f=\I(F):\OO_D \to \hh$ as follows:
 if $x=\alpha+J\beta=\Phi_J(z)\in \OO_D\cap \C_J$, then
\[f(x):=F_1(z)+JF_2(z).\]
The tensor product $\hh\otimes\C$ can be equipped with the complex structure induced by the second factor.
The slice function $f$ is then called \emph{slice-regular} if $F$ is holomorphic.
If a domain $\OO$ in $\hh$ is axially symmetric and intersects the real axis, then this definition of slice regularity is equivalent to the one proposed by Gentili and Struppa \cite{GeSt2007Adv}. We will denote by $\SR(\OO)$ the right quaternionic module of slice-regular functions on $\OO$.
If $f\in\SR(\hh)$, then $f$ is called a slice-regular \emph{entire} function. In particular, every polynomial $f(x)=\sum_{n=0}^dx^na_n\in\hh[x]$ with right quaternionic coefficients $a_n$, is an entire slice-regular function.

\subsection*{Operations on slice functions}

Let $\OO=\OO_D$ be an open axially symmetric domain in $\hh$. The \emph{slice product} of two slice functions $f=\I(F)$, $g=\I(G)$ on $\OO$ is defined by means of the pointwise product of the stem functions $F$ and $G$:
\[f\cdot g=\I(FG).\]
% The antiinvolution $f\mapsto f^c$ and the \emph{normal function} $N(f)$ are defined by
% \[f^c=\I(F^c),\qquad N(f)=f\cdot f^c\]
% respectively. 
% If $f,g\in\mc{SR}(\OO)$, then $f\cdot g$,  $f^c$ and  $N(f)\in\mc{SR}(\OO)$.
The function $f=\I(F)$ is called \emph{slice-preserving} if the $\hh$-components $F_1$ and $F_2$ of the stem function $F$ are real-valued.
This is equivalent to the condition $f(\overline x)=\overline{f(x)}$ for every $x\in\OO_D$.
If $f$ is slice-preserving, then $f\cdot g$ coincides with the pointwise product of $f$ and $g$. %In this case, we will denote it simply by $fg$. %but this is not true in general.
%The antiinvolution $f\mapsto f^c=\I(F^c)$ defines the \emph{normal function} $N(f)=f\cdot f^c$ of $f$. %, which is always slice-preserving. 
If $f,g$ are slice-regular on $\OO$, then also their slice product $f\cdot g$ is slice-regular on $\OO$. We recall that the slice product has also an interpretation in terms of pointwise quaternionic product:
\[
(f\cdot g)(x)=
\begin{cases}f(x)g(f(x)^{-1}xf(x))&\text{\ if $f(x)\ne0$,}\\0&\text{\ if $f(x)=0$}.
\end{cases}
\]
For every $f\in\SF(\OO)$, the slice function $f^c=\I(F^c)=\I(\overline F_1+i\overline F_2)$ is the \emph{slice-conjugate} of $f$ and $N(f)=f\cdot f^c=f^c\cdot f$  is the \emph{normal function} of $f$. If $f\in\SR(\OO)$, then also $f^c$, $N(f)$ do belong to $\SR(\OO)$. 

Assume that $f\in\SF(\OO)$ is not identically zero on $\OO$. 
Then the zero set $V(N(f))=\{x\in\OO\,|\,N(f)=0\}$ of the normal function does not coincide with the whole set $\OO$. 
The slice product permits us to introduce on $\OO\setminus V(N(f))$ the \emph{slice reciprocal} $f^{-\punto}$ of $f$, as the function $f^{-\punto}:=N(f)^{-1}\cdot f^c$ (see \cite[Prop.\ 2.4]{AlgebraSliceFunctions}). %such that $f\cdot f^{-\punto}=f^{-\punto}\cdot f=1$ on $\OO\setminus V(N(f))$.
Again, if $f$ is slice-regular, then $f^{-\punto}$ is slice-regular, too.

The \emph{slice derivatives} $\dd{f}{x},\dd{f\;}{x^c}$ of a slice functions $f=\I(F)$ are defined by means of the  Cauchy-Riemann operators applied to the inducing stem function $F$:
\[\dd{f}{x}=\I\left(\dd{F}{z}\right),\quad \dd{f\;}{x^c}=\I\left(\dd{F}{\overline z}\right).\]
Note that $f$ is slice-regular on $\OO$ if and only if $\dd{f\;}{x^c}=0$ and if $f$ is slice-regular on $\OO$ then also $\dd{f}{x}$ is slice-regular on $\OO$. Moreover, the slice derivatives satisfy the Leibniz product rule w.r.t.\ the slice product.

A slice-regular function $g$ is called \emph{slice-constant} if $\dd{g}x=0$ on $\OO$. We will denote by $\SC(\OO)$ the right $\hh$-module of slice-constant functions on $\OO$. If $\OO$ is a slice domain, then every $g\in\SC(\OO)$ is a quaternionic constant. If $\OO$ is a product domain, then other possibilities arise: 
for any imaginary unit $I\in\s$, a slice-constant function $\mu_I$ on $\hh\setminus\rr$ can be defined by setting, such as indicated in \cite[Remark 12]{GhPe_AIM}, 
\begin{equation}\label{eq:muI}
\mu_I(x):=\frac12\left(1+\frac{\IM(x)}{|\IM(x)|}I\right)\in\SR(\hh\setminus\R).
\end{equation}
By the representation formula (see \cite[Prop.\ 5]{GhPe_AIM} and \cite[Prop.\ 15]{Altavilla2018}), any slice-constant $g\in\SC(\hh\setminus\rr)$ is determined by its values on two arbitrarily chosen half-slices $\cc_J^+, \cc_K^+$, with $J\ne K$. If $g_{|\cc_J^+}=a_1\in\hh$ and $g_{|\cc_K^+}=a_2\in\hh$, then $g$ can be expressed as
\begin{equation}\label{eq:muJK}
g(x)=2\mu_{J}(x)J(J-K)^{-1}a_2-2\mu_{K}(x)K(J-K)^{-1}a_1.
\end{equation}

\section{Eigenvalue problems for slice functions}

\subsection*{Eigenvalue problem for slice-regular functions}
Let $\OO$ be an axially symmetric domain of $\hh$. Consider the following eigenvalue problem for the slice derivative operator $\dd{}x$ in the class of slice-regular functions:
\begin{equation}\label{eig:sr}
\begin{cases}\displaystyle
\cdf =f\lambda\text{\quad on $\OO$},\\
\text{with }f\in\SR(\OO),\ \lambda\in\hh.
\end{cases}
\end{equation}

\begin{remark}\label{rem:exp}
Let %$\lambda\in\hh\setminus\{0\}$ 
$\lambda\in\hh$ and let $e^{z\lambda}:=\sum_{n=0}^{+\infty} \frac{(z\lambda)^n}{n!}$. Then $e^{z\lambda}$ is a holomorphic stem function on $\C$, that induces the slice-regular entire function 
\[\exp_\lambda(x):=\I(e^{z\lambda})=\sum_{n=0}^{+\infty}\frac{x^n\lambda^n}{n!}.\] 
Observe that whenever $\lambda$ is not real, then in general $\exp_\lambda(x)\ne e^{x\lambda}:=\sum_{n=0}^{+\infty} \frac{(x\lambda)^n}{n!}$. The exponential function $\exp_\lambda(x)$ is slice-preserving if and only if $\lambda$ is real. In general, if $\lambda\in\cc_I$, then $\exp_\lambda(x)$ is \emph{one-slice-preserving}, i.e., it maps the slice $\cc_I$ into $\cc_I$.
%If $\lambda=0$, we have $\exp_0(x)\equiv1$. 

For any $\lambda\in\hh$, it holds $\dd {\exp_\lambda(x)}x=\exp_\lambda(x)\lambda$, since $\dd {e^{z\lambda}}z=e^{z\lambda}\lambda$, i.e., $f(x)=\exp_\lambda(x)$ is a solution of \eqref{eig:sr} on every $\OO$.
\end{remark}

We consider also a more general eigenvalue problem in which the solutions are searched in the space of slice functions on $\OO$ which are not necessarily slice-regular.

\subsection*{Eigenvalue problem for slice functions}
Let $\OO$ be an axially symmetric domain of $\hh$. Consider the following eigenvalue problem for the slice derivative operator $\dd{}x$ in the class $\SF^1(\OO)$ of slice functions induced by stem functions of the class $C^1$ on $\OO$:
\begin{equation}\label{eig:s}
\begin{cases}\displaystyle
\cdf =f\lambda\text{\quad on $\OO$},\\
\text{with }f\in\SF^1(\OO),\ \lambda\in\hh.
\end{cases}
\end{equation}

\begin{remark}
Let $g\in\SF^1(\OO)$. The function $f=g\cdot \exp_\lambda(x)$ obtained by taking the slice product of $g$ and $\exp_\lambda(x)$ is a solution of \eqref{eig:s} if and only if 
\[\cdf=\dd gx\cdot \exp_\lambda(x)+g\cdot \exp_\lambda(x)\lambda=f\lambda\Leftrightarrow\dd gx\cdot \exp_\lambda(x)=0\text{\quad on }\OO.
\]
%(here $f\cdot g$ denotes the \emph{slice product} of two slice functions $f,g\in\SF(\OO)$, i.e., $f\cdot g=\I(FG)$).\\
Therefore, for any anti-slice-regular function $g\in\SF^1(\OO)$ (i.e., such that $\dd gx=0$ on $\OO$), the function $f=g\cdot \exp_\lambda(x)$ is a solution of \eqref{eig:s}.  Moreover, it holds:
\begin{itemize}
\item
If $\OO$ is a slice domain (i.e., $\OO\cap\rr\ne\emptyset$) and $g$ is constant,
then  $f=g\cdot \exp_\lambda(x)\in\SR(\OO)$ is a solution of \eqref{eig:sr}.
\item
If $\OO$ is a product domain (i.e., $\OO\cap\rr=\emptyset$), then  $f=g\cdot \exp_\lambda(x)\in\SR(\OO)$ is a solution of \eqref{eig:sr} for every slice-constant function $g$. 
\end{itemize}
Conversely, if $f$ is a solution of \eqref{eig:s}, then set $g:=f\cdot (\exp_\lambda(x))^{-\punto}$. 
%where for any slice function $h\in\SF(\OO)$, $h^{-\punto}$ denotes the \emph{slice reciprocal} of $h$, defined on $\OO\setminus V(N(h))$ (see \cite[Prop.\ 2.4]{AlgebraSliceFunctions}). 
Since the normal function of $\exp_\lambda(x)$ is $N(\exp_\lambda(x))=e^{xt(\lambda)}$, where $t(\lambda)=\lambda+\overline\lambda$ is the \emph{trace} of $\lambda$, it holds $V(N(\exp_\lambda(x)))=\emptyset$. More precisely, it holds $(\exp_\lambda(x))^{-\punto}=\exp_{-\lambda}(x)$, since if $\lambda\in\cc_J$, then $\exp_\lambda(x)$ 
%$\in\SR_{\cc_J}(\hh)$
is a one-slice preserving function such that $\exp_\lambda(x)\cdot \exp_{-\lambda}(x)=1$ on the slice $\cc_J$, and therefore, by the representation formula (see \cite[Prop.\ 5]{GhPe_AIM}), the equality holds on the whole space $\hh$. We have
\begin{align*}
\dd gx&=\dd {(f\cdot \exp_{-\lambda}(x))}x=\cdf\cdot \exp_{-\lambda}(x)+f\cdot(-\exp_{-\lambda}(x)\lambda)\\
&=
(f\lambda)\cdot \exp_{-\lambda}(x)-f\cdot \exp_{-\lambda}(x)\lambda=0,
\end{align*}
and then $f=g\cdot\exp_\lambda(x)$, with $g\in \SF^1(\OO)$ anti-slice-regular.
\end{remark}

Denote by $\overline{\SR}(\OO)$ the right $\hh$-module of anti-slice-regular functions on $\OO$. Let $g=\I(G_1+\ui G_2)$. Then $g\in\overline{\SR}(\OO)$ if and only if the slice function $\overline g=\I(G_1-\ui G_2)$ is slice-regular, i.e., the stem function $G:=G_1+\ui G_2$ inducing $g$ is anti-holomorphic. We can summarize the previous remarks in the following statement.

\begin{proposition}\label{pro:sol_hom}
Let $\OO\subset\hh$ be an axially symmetric domain.
A function $f\in\SF^1(\OO)$ is a solution of \eqref{eig:s} if and only if $f=g\cdot \exp_\lambda(x)$, with $g\in\overline\SR(\OO)$ anti-slice-regular. The solution $f$ is slice-regular, and then a solution of \eqref{eig:sr}, if and only if $g\in\overline\SR(\OO)\cap\SR(\OO)=\SC(\OO)$, i.e., $g$ is slice-constant on $\OO$ (constant if $\OO$ is a slice domain).\qed
\end{proposition}

Let $R_\lambda$ be the operator of right multiplication by $\lambda$ and let $\mathcal D_\lambda=\dd{}x-R_\lambda$ denote the linear operator mapping a slice function $f\in\SF^1(\OO)$  to the continuous slice function 
\[\Dl f=\cdf -f\lambda.\]
Note that $\Dl f$ is slice-regular for any slice-regular function $f$. Another useful property of the operator $\Dl$ is the following: for every slice-constant $g\in\SC(\OO)$ and every $f\in\SF^1(\OO)$, it holds
\[\Dl(g\cdot f)=g\cdot \Dl f.
\]

\begin{remark}\label{rem:unique}
When $\OO$ is a slice domain, in view of Proposition \ref{pro:sol_hom} it holds 
\[\Ker (\Dl)\cap\SR(\OO)=\{c\cdot \exp_\lambda(x)\,|\, c\in\hh\}.
\]
If  $y\in\OO\cap\rr$, the solution of \eqref{eig:sr} is uniquely determined by its value at $y$, and takes the form $f=c\cdot \exp_\lambda(x)$, with $c=f(y)\exp_{-\lambda}(y)$. 

If $\OO$ is a product domain, then the solutions of \eqref{eig:sr} have more degrees of freedom (eight real d.o.f.\ instead of four).
% For any imaginary unit $I\in\s$, a slice-constant function $\mu_I$ on $\hh\setminus\rr$ can be defined by setting
% \[\mu_I(x):=\frac12\left(1+\frac{\IM(x)}{|\IM(x)|}I\right)\in\SR(\hh\setminus\R).
% \]
% Any slice-constant $g\in\SC(\hh\setminus\rr)$ is determined by its values on two arbitrarily chosen half-slices $\cc_J^+, \cc_K^+$, with $J\ne K$. If $g_{|\cc_J^+}=a_1\in\hh$ and $g_{|\cc_K^+}=a_2\in\hh$, then $g$ can be expressed as
% \[g(x)=2\mu_{J}(x)J(J-K)^{-1}a_2-2\mu_{K}(x)K(J-K)^{-1}a_1.
% \]
If $\mu_I\in\SC(\hh\setminus\rr)$ is the function defined in \eqref{eq:muI}, then formula \eqref{eq:muJK} for $J=-K=I$ implies  that $g:=\mu_I a_2+\mu_{-I}a_1$ is the unique slice-constant function on $\hh\setminus\rr$ with $g_{|\cc_I^+}=a_1$, $g_{|\cc_I^-}=a_2$. We then obtain, for every $I\in\s$, the representation
\[\Ker (\Dl)\cap\SR(\OO)=\{(\mu_I a_2)\cdot \exp_\lambda(x)+(\mu_{-I}a_1)\cdot \exp_\lambda(x)\,|\, a_1,a_2\in\hh\}
\]
for the kernel of the operator $\Dl$ on slice-regular functions on $\OO$. A function $f=g\cdot \exp_\lambda(x)=(\mu_I a_2)\cdot \exp_\lambda(x)+(\mu_{-I}a_1)\cdot \exp_\lambda(x)\in\Ker(\Dl)$ is uniquely determined by its values at two points, for example by $I$ and $-I$. Since $f\cdot \exp_{-\lambda}(x)=g$, a direct computation shows that
\[
a_1=\begin{cases}
f(I)\exp_{-\lambda}(f(I)^{-1}If(I))&\text{\quad if }f(I)\ne0,\\
0&\text{\quad if }f(I)=0,
\end{cases}
\]
and
\[
a_2=
\begin{cases}
f(-I)\exp_\lambda(f(-I)^{-1}If(-I))&\text{\quad if }f(-I)\ne0,\\
0&\text{\quad if }f(-I)=0,
\end{cases}
\]
where we used the fact that $\exp_{-\lambda}(-x)=\exp_\lambda(x)$ for every $x,\lambda\in\hh$.
\end{remark}

The eigenvalue problem \eqref{eig:sr} is the  homogeneous problem associated to the following generalized eigenvalue problem. Given $h\in\SR(\OO)$ and $\lambda\in\hh$, find $f\in\SR(\OO)$ such that
 \begin{equation}\label{eig:s_nonhom}
% \begin{cases}\displaystyle
 \Dl f=h\text{\quad on $\OO$}. %,\\
% \text{with }f,g\in\SR(\OO),\ \lambda\in\hh.
% \end{cases}
 \end{equation}

\begin{remark}\label{rem:01}
Assume that $\lambda\ne0$.
If $h(x)=x c_1+c_0$ is linear ($c_0,c_1\in\hh$), then 
$f(x)=-c_1\cdot (\lambda^{-2}+x\lambda^{-1})-c_0\lambda^{-1}$ is a solution of \eqref{eig:s_nonhom}, since
\begin{align*}
\Dl f&=-c_1\cdot \dd{(\lambda^{-2}+x\lambda^{-1})}x+c_1\cdot (\lambda^{-2}+x\lambda^{-1})\lambda+c_0=\\
&=-c_1\lambda^{-1}+c_1\cdot(\lambda^{-1}+x)+c_0=c_1\cdot x+c_0=xc_1+c_0=h(x).\\
\end{align*}

In view of the linearity of \eqref{eig:s_nonhom}, the previous solution is uniquely determined up to solutions of the homogeneous problem \eqref{eig:sr}. From Proposition \ref{pro:sol_hom} we can infer that the general solution is
\[f(x)=a_0\cdot \exp_\lambda(x)-c_1\cdot (\lambda^{-2}+x\lambda^{-1})-c_0\lambda^{-1},
\]
for any $a_0\in\hh$. On product domains the general solutions have similar forms, with any slice-constant function $g\in\SC(\OO)$ in place of the constant $a_0$.
\end{remark}

Remark \ref{rem:01} is a hint for finding solutions of \eqref{eig:s_nonhom} for every polynomial right-hand side $h\in\hh[x]$ and $\lambda\ne0$. In this case we can take for $\OO$ the whole quaternionic space $\hh$ (a slice domain) or the product domain $\hh\setminus\rr$.

\begin{proposition}\label{pro:polynomial}
Let $\lambda\ne0$ and let $h(x)=\sum_{n=0}^dx^nc_n\in\hh[x]$. For any $n\in\nn$, let $f_n\in\hh[x]$ be the slice-regular polynomial
\[f_n(x)=-c_n\cdot \left(\sum_{k=0}^n \frac{x^k\lambda^k}{k!}\right)n!\lambda^{-n-1}.
\]
%(cf. incomplete Gamma function)
Then the polynomial $f_*=\sum_{n=0}^df_n$ is a solution of \eqref{eig:s_nonhom} with right-hand side $h$. The general solution of \eqref{eig:s_nonhom} on $\OO$ is $f=f_*+g\cdot \exp_\lambda(x)$, with $g\in\SC(\OO)$. 
If $\OO=\hh$, a slice domain, and $y\in\rr$, then the solution of \eqref{eig:s_nonhom} %for a polynomial $h$ 
is uniquely determined by its value at $y$, and takes the form $f=f_*+c\cdot \exp_\lambda(x)$, with $c=(f(y)-f_*(y))\exp_{-\lambda}(y)\in\hh$. If $\OO=\hh\setminus\rr$, a product domain, then the solution of \eqref{eig:s_nonhom} 
is uniquely determined by its values at two points.
\end{proposition}
\begin{proof}
We have 
\begin{align*}
\Dl f_n(x)&=-c_n\cdot \dd{}x\left(\sum_{k=0}^n \frac{x^k\lambda^k}{k!}\right)n!\lambda^{-n-1}+c_n\cdot \left(\sum_{k=0}^n \frac{x^k\lambda^k}{k!}\right)n!\lambda^{-n}=\\
&=-c_n\cdot \left(\sum_{k=1}^n \frac{x^{k-1}\lambda^{k-1}}{(k-1)!}\right)n!\lambda^{-n}+c_n\cdot \left(\sum_{k=0}^n \frac{x^k\lambda^k}{k!}\right)n!\lambda^{-n}=c_n\cdot x^n=x^nc_n.
\end{align*}
Therefore, $\Dl f=\Dl f_*=h$, if $f=f_*+g\cdot \exp_\lambda(x)$, with $g\in\SC(\OO)$. 
The last statement follows immediately from Proposition \ref{pro:sol_hom} and from Remark \ref{rem:unique}.
\end{proof}

The preceding Proposition means that one can define a right inverse of $\Dl$ on the space of quaternionic polynomials of degree at most $d$. It is the operator $\Sl:\hh_d[x]\to\hh_d[x]$ defined as follows: if $p(x)=\sum_{n=0}^dx^na_n$, then
\begin{equation}\label{def:Sl}
\Sl p=-\sum_{n=0}^d a_n\cdot \left(\sum_{k=0}^n\frac{x^k\lambda^k}{k!}\right)n!\lambda^{-n-1}=\sum_{k=0}^dx^kb_k,
\end{equation}
with
\[
b_k=-\sum_{n=k}^d\frac{a_n\lambda^{k-n-1}n!}{k!}\text{\quad for $k=0,\ldots,d$}.
\]

Now we consider equation \eqref{eig:s_nonhom} with an exponential right-hand side $h(x)=g\cdot \exp_\mu(x)$, with $\mu\in\hh$ and $g$ slice-constant. 

\begin{proposition}\label{pro:lm}
Let $\lambda,\mu\in\hh$, with $0<|\mu|<|\lambda|$. For every $n\in\nn$, let $f_n^{\mu,\lambda}\in\hh[x]$ be the slice-regular polynomial of degree $n$
\[f_n^{\mu,\lambda}(x):=-\sum_{k=0}^n\frac{x^k\mu^n\lambda^k}{k!}\lambda^{-n-1}.
\]
Then the series of functions $\sum_{n=0}^{+\infty} f_n^{\mu,\lambda}(x)$ converges uniformly on the compact sets of $\hh$ to an entire slice-regular function $f^{\mu,\lambda}\in\SR(\hh)$ such that $\Dl (f^{\mu,\lambda})=\exp_\mu(x)$. As a consequence,
$\Dl(g\cdot f^{\mu,\lambda})=g\cdot \exp_\mu(x)$ for every slice-constant $g$ on $\OO$.
If $\lambda$ and $\mu$ commute, then the function $f^{\mu,\lambda}$ has the expected explicit form $f^{\mu,\lambda}(x)=\exp_\mu(x)(\mu-\lambda)^{-1}$.
\end{proposition}
\begin{proof}
Since $f_n^{\mu,\lambda}(x):=-\mu^n\cdot\sum_{k=0}^n\frac{x^k\lambda^k}{k!}\lambda^{-n-1}$, as in the proof of Proposition \ref{pro:polynomial} we get that
\[\Dl(f_n^{\mu,\lambda})=x^n\frac{\mu^n}{n!}.
\]
From the estimate
\[|f_n^{\mu,\lambda}(x)|\le \sum_{k=0}^n \frac {|x|^k |\mu|^n|\lambda|^k}{k!}|\lambda|^{-n-1}=\sum_{k=0}^n \frac {|x\lambda|^k}{k!}\frac{|\mu|^n}{|\lambda|^n}|\lambda|^{-1}\le \frac{e^{|x\lambda|}}{|\lambda|} \frac{|\mu|^n}{|\lambda|^n}
\]
follows the uniform convergence of the series on compacts and then the equality $\Dl f^{\mu,\lambda}(x)=\sum_{n=0}^{+\infty} x^n\mu^n/n!=\exp_\mu(x)$.
To prove the last statement, we can assume that $\lambda,\mu\in\cc_I$ for a unit $I\in\s$ and write
\[f^{\mu,\lambda}(x)=\sum_{n=0}^{+\infty} f_n^{\mu,\lambda}(x)=\sum_{k=0}^{+\infty} \left(-\sum_{n=k}^{+\infty} \frac{x^k\mu^n}{k!}\lambda^{-n-1+k} \right).
\]
We then compute the sums 
\[g_k^{\mu,\lambda}(x):=-\sum_{n=k}^{+\infty} \frac{x^k\mu^n}{k!}\lambda^{-n-1+k}=
-\frac{x^k}{k!}\lambda^{-1+k}\sum_{n=k}^{+\infty} (\mu\lambda^{-1})^n=\frac{x^k}{k!}\mu^k(\mu-\lambda)^{-1}\]
and conclude observing that
\[f^{\mu,\lambda}(x)=\sum_{k=0}^{+\infty} g_k^{\mu,\lambda}(x)=\sum_{k=0}^{+\infty} \frac{x^k}{k!}\mu^k(\mu-\lambda)^{-1}=\exp_\mu(x)(\mu-\lambda)^{-1}.
\]
\end{proof}

%\begin{remark}\label{rem:commute}
If $\lambda$ and $\mu$ commute and $\lambda\ne\mu$, then $\exp_\mu(x)(\mu-\lambda)^{-1}$ solves $\Dl f=\exp_\mu(x)$ even if $|\mu|\ge|\lambda|$ or if one of the parameters $\lambda,\mu$ vanishes.

If $\lambda=\mu$, then a solution of the equation $\Dl f=\exp_\lambda(x)$ is easily found from the correspondent complex problem. The one-slice-preserving slice-regular function $f(x)= x\exp_\lambda(x)$ satisfies the equation $\Dl f=\exp_\lambda(x)$. 
%\end{remark}

In the next Proposition we show how to remove the assumption $0<|\mu|<|\lambda|$ of Proposition \ref{pro:lm} and find another slice-regular solution of the equation $\Dl f=\exp_\mu(x)$. More precisely, we find the power series expansion around the origin of the unique solution vanishing at $0$.

% \begin{remark}\label{rem:commute}
% If $\lambda$ and $\mu$ commute and $\lambda\ne\mu$, then $\exp_\mu(x)(\mu-\lambda)^{-1}$ solves $\Dl f=\exp_\mu(x)$ even if $|\mu|\ge|\lambda|$ or if one of the parameters $\lambda,\mu$ vanishes.

% If $\lambda=\mu$, then a solution of the equation $\Dl f=\exp_\lambda(x)$ is easily found from the correspondent complex problem. The one-slice-preserving slice-regular function $f(x)= x\exp_\lambda(x)$ satisfies the equation $\Dl f=\exp_\lambda(x)$. 
% \end{remark}

\begin{proposition}\label{pro:lmII}
Let $\lambda,\mu\in\hh$. %, with $\lambda\ne\mu$. 
Let $g^{\mu,\lambda}$ be the sum of the slice-regular power series
\[g^{\mu,\lambda}(x):=\sum_{n=1}^{+\infty}\frac{x^n}{n!}\sum_{k=0}^{n-1}\mu^k\lambda^{n-k-1}.
\]
The series converges uniformly on the compact sets of $\hh$ to an entire slice-regular function such that $\Dl (g^{\mu,\lambda})=\exp_\mu(x)$ and $g^{\mu,\lambda}(0)=0$. As a consequence, 
$\Dl(h\cdot g^{\mu,\lambda})=h\cdot \exp_\mu(x)$ for every slice-constant $h$ on $\OO$.
If $\lambda$ and $\mu$ commute and $\lambda\ne\mu$, then the function $g^{\mu,\lambda}$ has the expected explicit form $g^{\mu,\lambda}(x)=(\exp_\mu(x)-\exp_\lambda(x))(\mu-\lambda)^{-1}$.
If $\lambda=\mu$, then $g^{\lambda,\lambda}(x)=x\exp_\lambda(x)$.
\end{proposition}
\begin{proof}
Assume for the moment that the series is uniformly convergent. Then
\begin{align*}
\Dl g^{\mu,\lambda}&=\dd{}{x}\left(\sum_{n=1}^{+\infty}\frac{x^n}{n!}\sum_{k=0}^{n-1}\mu^k\lambda^{n-k-1}\right)-\left(\sum_{n=1}^{+\infty}\frac{x^n}{n!}\sum_{k=0}^{n-1}\mu^k\lambda^{n-k-1}\right)\lambda\\
&=\sum_{n=1}^{+\infty}\frac{nx^{n-1}}{n!}\sum_{k=0}^{n-1}\mu^k\lambda^{n-k-1}-\sum_{n=1}^{+\infty}\frac{x^n}{n!}\sum_{k=0}^{n-1}\mu^k\lambda^{n-k}\\
&=\sum_{n=0}^{+\infty}\frac{x^{n}}{n!}\sum_{k=0}^{n}\mu^k\lambda^{n-k}-\sum_{n=1}^{+\infty}\frac{x^n}{n!}\sum_{k=0}^{n-1}\mu^k\lambda^{n-k}\\
&=1+\sum_{n=1}^{+\infty}\frac{x^n}{n!}\mu^n=\exp_\mu(x).
\end{align*}
Let \[g_n(x):=\frac{x^n}{n!}\sum_{k=0}^{n-1}\mu^k\lambda^{n-k-1}.
\]
If $|\lambda|\ne|\mu|$, then
\[
|g_n(x)|\le\frac{|x|^n}{n!}\sum_{k=0}^{n-1}|\mu|^k|\lambda|^{n-k-1}=
\frac{|x|^n}{n!}\frac{|\lambda|^n-|\mu|^n}{|\lambda|-|\mu|}.
\]
Since the estimate is symmetric in $\lambda$ and $\mu$, we can assume that $|\mu|<|\lambda|$. Then 
\[|g_n(x)|\le\frac{|2x\lambda|^n}{n!}\frac{2^{-n}}{|\lambda|-|\mu|}\le e^{|2x\lambda|}\frac{2^{-n}}{|\lambda|-|\mu|}
\]
from which follows the uniform convergence of the series $g^{\mu,\lambda}(x)=\sum_{n\ge1}g_n(x)$ on the compacts sets of $\hh$.  
If $|\lambda|=|\mu|$, then
\[
|g_n(x)|\le\frac{|x|^n}{n!}\sum_{k=0}^{n-1}|\lambda|^{n-1}=
\frac{|x|^n}{(n-1)!}|\lambda|^{n-1}= \frac{|x||2x\lambda|^{n-1}}{(n-1)!}2^{-n+1}
\le |x|e^{|2x\lambda|}2^{-n+1}
\]
and the uniform convergence of the series follows also in this case. If $\lambda$ and $\mu$ commute, the following equality holds:
\[
g^{\mu,\lambda}(x)=\sum_{n=1}^{+\infty}\frac{x^n}{n!}\sum_{k=0}^{n-1}\mu^k\lambda^{n-k-1}=\sum_{n=1}^{+\infty}\frac{x^n}{n!}(\mu^n-\lambda^n){(\mu-\lambda)^{-1}}.
\]
The last statement follows immediately from the series expansion
\[
g^{\lambda,\lambda}(x)=\sum_{n=1}^{+\infty}\frac{x^n}{n!}n\lambda^{n-1}=x\sum_{n=0}^{+\infty}\frac{x^n\lambda^{n}}{n!}=x\exp_\lambda(x).
\]
\end{proof}

\begin{remark}
In view of Proposition \ref{pro:sol_hom}, the solutions $f^{\mu,\lambda}$ and $g^{\mu,\lambda}$ of Propositions \ref{pro:lm} and \ref{pro:lmII} of the equation $\Dl f=\exp_\mu(x)$, as soon as they are both defined, just differ by a function of the form $h\cdot \exp_\lambda(x)$, with $h\in\SC(\OO)$. Indeed, it holds $f^{\mu,\lambda}-g^{\mu,\lambda}=f^{\mu,\lambda}(0)\cdot \exp_\lambda(x)$. 
\end{remark}

\begin{examples}\label{ex1}
(1)\quad Let $\mu=i$, $\lambda=j$. A direct computation shows that the solution $g^{i,j}$ of the equation $\mathcal D_jf=\exp_i(x)=\cos x+(\sin x) i$ given by Proposition \ref{pro:lmII} has the explicit expression
\[
g^{i,j}(x)=\sum_{n=1}^{+\infty}\frac{x^n}{n!}\sum_{k=0}^{n-1}i^k j^{n-k-1}=\frac12\left(\sin x+x\cos x+(x\sin x)(i+j)+(\sin x-x\cos x)k\right).
\]
Observe that the slice-regular functions $\cos x=\I(\cos z)$ and $\sin x=\I(\sin z)$ are exactly the functions defined in \cite[Def.11.23]{GHS} in the more general context of Clifford algebras.

(2)\quad Let $\mu=i$, $\lambda=2j$. Then the solution $g^{i,2j}$ of the equation $\mathcal D_{2j}f=\exp_i(x)$ is
\[
g^{i,2j}(x)=\frac13\left(2\sin(2x)-\sin x+(\cos x-\cos(2x))(i+2j)+(2\sin x-\sin(2x))k\right).
\]

(3)\quad Let $\mu=2i$, $\lambda=j$. Then the solution $g^{2i,j}$ of the equation $\mathcal D_{j}f=\exp_{2i}(x)$ is the function
\[
g^{2i,j}(x)=\frac13\left(2\sin(2x)-\sin x+(\cos x-\cos(2x))(2i+j)+(2\sin x-\sin(2x))k\right).
\]
\end{examples}

\subsection*{Eigenvalue problem of the second order for slice-regular functions}
Propositions \ref{pro:sol_hom} and \ref{pro:lmII} permit us to study eigenvalue problems of the second order. Given $\lambda,\mu\in\hh$, we consider the equation
\begin{equation}\label{eq:eig2}
\begin{cases}\displaystyle
\Dm\Dl f=0\text{\quad on $\OO$},\\
\text{with }f\in\SR(\OO).
\end{cases}
\end{equation}
Notice that two operators $\Dl$ and $\Dm$ commute if and only if $\lambda$ and $\mu$ commute, i.e., they belong to the same complex slice $\cc_I$, $I\in\s$.

\begin{corollary}\label{cor:order2}
Let $\lambda,\mu\in\hh$. 
\begin{enumerate}
	\item
If $\lambda\ne\mu$, then the function $g^{\mu,\lambda}$ of Proposition \ref{pro:lmII} is a solution of \eqref{eq:eig2}. Any solution of \eqref{eq:eig2} is of the form $f=h_1\cdot g^{\mu,\lambda}+h_2\cdot\exp_\lambda(x)$, with $h_1,h_2\in\SC(\OO)$. 
\item
If  $\lambda$ and $\mu$ are distinct and commute, then the solutions of \eqref{eq:eig2} are the functions of the form $f=h_1\cdot \exp_\mu(x)+h_2\cdot\exp_\lambda(x)$, with $h_1,h_2\in\SC(\OO)$. 
\item
If $\lambda=\mu$, then the solutions of \eqref{eq:eig2}, i.e., the functions  $f\in\SR(\OO)$ such that $\Dl^2 f=0$, are the functions of the form $f=h_1\cdot \exp_\lambda(x)+h_2\cdot(x\exp_\lambda(x))$, with $h_1,h_2\in\SC(\OO)$.
\end{enumerate}
\end{corollary}
\begin{proof}
From Proposition \ref{pro:lmII} and Proposition \ref{pro:sol_hom}, $\Dm\Dl(g^{\mu,\lambda})=\Dm\exp_\mu(x)=0$. Conversely, if $f$ is a solution of \eqref{eq:eig2}, then $\Dl f$ has the form $h_1\cdot\exp_\mu(x)$, with $h_1\in\SC(\OO)$, and then $f-h_1\cdot g^{\mu,\lambda}\in\Ker(\Dl)\cap\SR(\OO)$. From Proposition \ref{pro:sol_hom}, we get $f-h_1\cdot g^{\mu,\lambda}=h_2\cdot\exp_\lambda(x)$, with $h_2\in\SC(\OO)$. 

If $\lambda$ and $\mu$ are distinct but commuting elements, then Proposition \ref{pro:lmII} gives the solution $g^{\mu,\lambda}(x)=(\exp_\mu(x)-\exp_\lambda(x))(\mu-\lambda)^{-1}$ of $\Dl f=\exp_\mu(x)$.
Moreover, 
\[\exp_\mu(x)(\mu-\lambda)^{-1}=\sum_{k=0}^{+\infty}\frac{x^k\mu^k}{k!}(\mu-\lambda)^{-1}=\sum_{k=0}^{+\infty}\frac{x^k(\mu-\lambda)^{-1}\mu^k}{k!}=(\mu-\lambda)^{-1}\cdot\exp_\mu(x)
\]
and similarly for $\exp_\lambda(x)$. Then $h_1\cdot g^{\mu,\lambda}=(h_1(\mu-\lambda)^{-1})\cdot\exp_\mu(x)-(h_2(\mu-\lambda)^{-1})\cdot\exp_\lambda(x)=h'_1\cdot\exp_\mu(x)+h'_2\cdot\exp_\lambda(x)$, with $h'_1,h'_2\in\SC(\OO)$. 
The statement of the third case addressing $\lambda=\mu$ follows from Proposition \ref{pro:lmII} and Proposition \ref{pro:sol_hom}.
\end{proof}

\subsection*{Eigenvalue problem of the $m$th-order for slice-regular functions}
Now we generalize Corollary \ref{cor:order2} to eigenvalue problems of any order $m$. Let $\lambda_1,\ldots,\lambda_m\in\hh$ and consider the equation
\begin{equation}\label{eq:eig_m}
\begin{cases}\displaystyle
\DL{1}\cdots\DL{m} f=0\text{\quad on $\OO$},\\
\text{with }f\in\SR(\OO).
\end{cases}
\end{equation}

We begin by introducing a family of slice-regular functions which generalize the exponential functions $\exp_\lambda(x)$ of Remark \ref{rem:exp}.

\begin{definition}\label{def:E}
Let $\lambda_1,\ldots,\lambda_m\in\hh$, with $m\ge1$. The \emph{generalized exponential function} associated with $\Lambda=(\lambda_1,\ldots,\lambda_m)$ is the slice-regular entire function $E_\Lambda(x)$ defined by
\[
E_\Lambda(x):=\sum_{n=m-1}^{+\infty}\frac{x^{n}}{n!}\sum_{|K|=n-m+1}\lambda_1^{k_1}\cdots\lambda_m^{k_m},
\]
where the sum is extended over the multi-indices $K=(k_1,\ldots,k_m)$ of non-negative integers such that $|K|=k_1+\cdots+k_m=n-m+1$. In particular, it holds $E_{(\lambda)}=\exp_{\lambda}(x)$ for every $\lambda\in\hh$.
\end{definition}

\begin{proposition}\label{pro:E}
The series in the definition of $E_\Lambda(x)$ converges uniformly on the compacts of $\hh$, i.e., the entire slice-regular function $E_\Lambda$  in the previous definition is well-defined. Moreover, if $\Lambda=(\lambda_1,\ldots,\lambda_m)=(\lambda,\ldots,\lambda)$, then $E_{\Lambda}(x)=\frac{x^{m-1}}{(m-1)!}\exp_\lambda(x)$. It holds $E_\Lambda(0)=1$ if $\Lambda=(\lambda)$, otherwise $E_\Lambda(0)=0$.
\end{proposition}
\begin{proof}
Let $m$ be fixed, $n\ge m-1$ and define 
\[e_n(x):=\frac{x^n}{n!}\sum_{|K|=n-m+1}\lambda_1^{k_1}\cdots\lambda_m^{k_m}.
\]
Let $a:=\max\{|\lambda_1|,\ldots,|\lambda_m|\}$. Then we have 
\begin{align*}
|e_n(x)|&\le\frac{|x|^n}{n!}\sum_{|K|=n-m+1}|\lambda_1|^{k_1}\cdots|\lambda_m|^{k_m}
\le\frac{|x|^n}{n!}\sum_{|K|=n-m+1}a^{n-m+1}\\
&\le\frac{|x|^n}{n!}\binom{n}{n-m+1}a^{n-m+1}
=\frac{|x|^na^{n-m+1}}{(n-m+1)!(m-1)!}\\
&=\frac{|2x|^{m-1}}{(m-1)!}\frac{|2ax|^{n-m+1}}{(n-m+1)!}2^{-n}
\le \frac{|2x|^{m-1}}{(m-1)!}e^{|2ax|}2^{-n},
\end{align*}
from which follows the uniform convergence of the series $E_\Lambda(x)=\sum_{n\ge m-1}e_n(x)$ on the compacts sets of $\hh$.

If $\lambda_1=\cdots=\lambda_m=\lambda$, then 
\begin{align*}
E_{\Lambda}(x)&=\sum_{n=m-1}^{+\infty}\frac{x^{n}}{n!}\sum_{|K|=n-m+1}\lambda^{n-m+1}=
\sum_{n=m-1}^{+\infty}\frac{x^{n}}{n!}\binom{n}{n-m+1}\lambda^{n-m+1}\\
&=x^{m-1}\sum_{n=m-1}^{+\infty}\frac{x^{n-m+1}}{(n-m+1)!(m-1)!}\lambda^{n-m+1}
=\frac{x^{m-1}}{(m-1)!}\exp_\lambda(x).
\end{align*}
\end{proof}

We set $g_1(x):=\exp_{\lambda_1}(x)$ and we define $g_2(x)$ as the solution $g^{\lambda_1,\lambda_2}(x)$ of the equation $\DL 2g_2=g_1$ obtained in Proposition \ref{pro:lmII}. Then $\DL 1\DL 2g_2=\DL 1g_1=0$, as in Corollary \ref{cor:order2}. 
Observe that $g_2(x)$ is a generalized exponential function, since $g_2(x)=E_{(\lambda_1,\lambda_2)}(x)$. 
We now complete the construction of a sequence of entire slice-regular functions $g_1,\ldots g_m$ such that $\DL {k+1}g_{\ell+1}=g_\ell$ for every $\ell=1,\ldots,m-1$. The last element of the sequence, the function $g_m$, is then a solution of equation \eqref{eq:eig_m}, with the additional property that $\DL{2}\cdots\DL{m} g_m=g_1\ne0$.

For every $\ell=2,\ldots,m$, we put 
\[
g_\ell(x):=E_{(\lambda_1,\ldots,\lambda_\ell)}(x)
=\sum_{n=\ell-1}^{+\infty}\frac{x^{n}}{n!}\sum_{|K|=n-\ell+1}\lambda_1^{k_1}\cdots\lambda_\ell^{k_\ell}.
\]

\begin{proposition}\label{pro:Dm}
It holds $\DL{\ell+1} g_{\ell+1}=g_{\ell}$ for every $\ell=1,\ldots,m-1$.
\end{proposition}
\begin{proof}
\begin{align*}
\DL{\ell+1} g_{\ell+1}&=\dd{}{x}\left(\sum_{n=\ell}^{+\infty}\frac{x^{n}}{n!}\sum_{|K|=n-\ell}\lambda_1^{k_1}\cdots\lambda_{\ell+1}^{k_{\ell+1}}\right)-\left(\sum_{n=\ell}^{+\infty}\frac{x^{n}}{n!}\sum_{|K|=n-\ell}\lambda_1^{k_1}\cdots\lambda_{\ell+1}^{k_{\ell+1}}\right)\lambda_{\ell+1}\\
&=\sum_{n=\ell}^{+\infty}\frac{nx^{n-1}}{n!}\sum_{|K|=n-\ell}\lambda_1^{k_1}\cdots\lambda_{\ell+1}^{k_{\ell+1}}-\sum_{n=\ell}^{+\infty}\frac{x^n}{n!}\sum_{|K|=n-\ell}\lambda_1^{k_1}\cdots\lambda_{\ell+1}^{k_{\ell+1}+1}\\
&=\sum_{n=\ell-1}^{+\infty}\frac{x^{n}}{n!}\sum_{|K|=n-\ell+1}\lambda_1^{k_1}\cdots\lambda_{\ell+1}^{k_{\ell+1}}-\sum_{n=\ell}^{+\infty}\frac{x^n}{n!}\sum_{|K|=n-\ell}\lambda_1^{k_1}\cdots\lambda_{\ell+1}^{k_{\ell+1}+1}\\
&=\frac{x^{\ell-1}}{(\ell-1)!}+\sum_{n=\ell}^{+\infty}\frac{x^{n}}{n!}\sum_{|K|=n-\ell+1}\lambda_1^{k_1}\cdots\lambda_{\ell}^{k_{\ell}}=g_\ell(x).
\end{align*}
\end{proof}

\begin{corollary}\label{cor:order_m}
Let $\lambda_1,\ldots,\lambda_m\in\hh$. 
Let $g_{\lambda_m}(x):=\exp_{\lambda_m}(x)=E_{(\lambda_m)}(x)$. 
For $m\ge2$ and $1\le i< m$, let us denote by $g_{\lambda_i,\ldots,\lambda_m}$ the solution of the equation $\DL{m}g_{\lambda_i,\ldots,\lambda_m}=g_{\lambda_i,\ldots,\lambda_{m-1}}$
given by Proposition \ref{pro:Dm}, i.e., $g_{\lambda_i,\ldots,\lambda_m}=E_{(\lambda_i,\ldots,\lambda_m)}$.
Then the generalized exponential function $g_m=E_{(\lambda_1,\ldots,\lambda_m)}$ is a solution of \eqref{eq:eig_m} and any solution of \eqref{eq:eig_m} on $\OO$ is of the form 
\begin{equation}\label{eq:ker}
f=\sum_{i=1}^mh_i\cdot E_{(\lambda_i,\ldots,\lambda_m)}
\end{equation}
with $h_i\in\SC(\OO)$ for every $i$.
\end{corollary}
\begin{proof}
From Proposition \ref{pro:Dm} we infer that $\DL{1}\cdots\DL{m}g_m=\DL{1}\cdots\DL{m-1}g_{m-1}=\cdots=\DL{1}g_1=0$. 
We prove \eqref{eq:ker} by induction over $m$. The cases $m=1,2$ have already been considered in Proposition \ref{pro:sol_hom} and in Corollary \ref{cor:order2}. Let $m\ge 3$ and assume that $f$ is a solution of \eqref{eq:eig_m}.
Then $\DL mf$ solves the equation $\DL{1}\cdots\DL{m-1}g=0$. By the induction hypothesis, $\DL mf$ has the form
\[
\DL mf= \sum_{i=1}^{m-1}h_i\cdot E_{(\lambda_i,\ldots,\lambda_{m-1})}
\]
with $h_i\in\SC(\OO)$ for every $i$. Therefore, $f-\sum_{i=1}^{m-1}h_i\cdot g_{\lambda_i,\ldots,\lambda_{m}}\in\Ker(\DL{m})$. From Proposition \ref{pro:sol_hom}, we get $f-\sum_{i=1}^{m-1}h_i\cdot g_{\lambda_i,\ldots,\lambda_{m}}=h_m\cdot\exp_{\lambda_m}(x)=h_m\cdot E_{(\lambda_m)}(x)$, with $h_m\in\SC(\OO)$. This proves \eqref{eq:ker} for every $m$.
\end{proof}

\begin{remarks}\label{rem:m}
(1)\quad 
If $\lambda_1=\cdots=\lambda_m=\lambda$, then the solutions of equation \eqref{eq:eig_m}, that now takes the form $(\Dl)^mf=0$, can be deduced directly from the complex case. If $\lambda\in\cc_I$, the functions $x^k\exp_{\lambda}(x)$ are one-slice preserving slice-regular, and then for any $k=0,\ldots,m-1$ the complex solutions $x^k\exp_{\lambda}(x)$ for $x\in\cc_I$ extend slice-regularly in a unique way to $\OO$, giving the general solution
\[
f=\sum\limits_{k=0}^{m-1} h_k\cdot (x^{k}\exp_{\lambda}(x))
\]
on $\OO$, with $h_k\in\SC(\OO)$. The same result can be obtained from Corollary \ref{cor:order_m}, since by Proposition \ref{pro:E} it holds $E_{(\lambda_i,\ldots,\lambda_m)}(x)=\frac{x^{m-i}}{(m-i)!}\exp_{\lambda}(x)$ for every $i=1,\ldots,m$.
\medskip

(2)\quad
%\label{rem:m_equal}
Given $c_1,\ldots,c_m\in\hh$, the combination of generalized exponentials \[
f=\sum_{i=1}^mc_i\cdot E_{(\lambda_i,\ldots,\lambda_m)}
=\sum_{i=1}^m\sum_{n=m-i}^{+\infty}\frac{x^{n}c_i}{n!}\sum_{|K|=n-m+i}\lambda_i^{k_i}\cdots\lambda_m^{k_m}.
\]
is the unique entire solution of equation \eqref{eq:eig_m} such that
\[
\DL{i+1}\cdots\DL{m}f(0)=c_i\text{\quad for every }i=1,\ldots,m-1,\text{ and } f(0)=c_m.
\]
The \emph{canonical solution} $E_{(\lambda_1,\ldots,\lambda_m)}\not\in\ker(\DL{2}\cdots\DL{m})$ corresponds to the choice of parameters $c_1=1$, $c_2=\cdots=c_{m}=0$.

(3)\quad
%\begin{remark}\label{rem:m_commute}
If all the eigenvalues $\lambda_1,\ldots,\lambda_m$ commute and are distinct, then a simple computation shows that 
\[
f_m(x)=\exp_{\lambda_1}(x)\prod_{j=2}^m(\lambda_1-\lambda_j)^{-1}
\]
satisfies $\DL m f_m=f_{m-1}$. Then $f_m$ is another solution of \eqref{eq:eig_m}. Indeed, using the preceding remark it can be verified that under the commutativity assumption the following equality holds
\[
f_m(x)=
\sum_{i=1}^m\prod_{j=2}^i(\lambda_1-\lambda_j)^{-1}\cdot E_{(\lambda_i,\ldots,\lambda_m)}(x)=\sum_{i=1}^m E_{(\lambda_i,\ldots,\lambda_m)}(x)\prod_{j=2}^i(\lambda_1-\lambda_j)^{-1},
\]
in accordance with formula \eqref{eq:ker} of Corollary \ref{cor:order_m}. In the preceding formula the product is assumed to be equal to $1$ when $i=1$. Since the operators $\DL1,\ldots,\DL{m}$ commute, the general form of the solutions of \eqref{eq:eig_m} can be written in the form 
\[
f=\sum_{k=1}^m h_k\cdot \exp_{\lambda_k}(x)
\]
with $h_k\in\SC(\OO)$ for every $k$. Since $\lambda_1,\ldots,\lambda_m$ belong to a complex slice $\cc_I$ (unique is the eigenvalues are not all real), this result can be obtained also from the complex case.
%\end{remark}
\end{remarks}

\begin{examples}
(1)\quad 
A solution of $\mathcal D_{1+i}\mathcal D_{1+j}f=0$ is given by the generalized exponential
\[
E_{(1+i,1+j)}(x)=\sum_{n=1}^{+\infty}\frac{x^{n}}{n!}\sum_{|K|=n-1}(1+i)^{k_1}(1+j)^{k_2}.
\]
A direct computation shows that 
\[
E_{(1+i,1+j)}(x)=\frac12 e^x\left((x\cos x+\sin x) +x\sin x(i+j)+(-x\cos x+\sin x)k\right).
\]
It holds $E_{(1+i,1+j)}(0)=0$, $\mathcal D_{1+j}E_{(1+i,1+j)}(0)=E_{(1+i)}(0)=\exp_{1+i}(0)=1$. 
\medskip

(2)\quad
The canonical solution of $\mathcal D_i\mathcal D_j\mathcal D_kf=0$ is given by the generalized exponential
\[
E_{(i,j,k)}(x)=\sum_{n=2}^{+\infty}\frac{x^{n}}{n!}\sum_{|K|=n-2}i^{k_1}j^{k_2}k^{k_3}.
\]
A direct computation gives
\begin{align*}
E_{(i,j,k)}(x)&=\frac18\left(x(x+3)\cos x+(x^2+3x-3)\sin x\right)\\
&\quad +\frac18\left(-x(x+1)\cos x+(x^2+x+1)\sin x\right)(i+k)\\
&\quad +\frac18\left(x(x-1)\cos x+(x^2-x+1)\sin x\right)j.
\end{align*}
Moreover, it holds 
\[
E_{(i,j,k)}(0)=0,\ \mathcal D_{k}E_{(i,j,k)}(0)=E_{(i,j)}(0)=0,\ \mathcal D_{j}\mathcal D_{k}E_{(i,j,k)}(0)=E_{(i)}(0)=1.
\]

(3)\quad 
The generalized exponential functions also solve eigenvalue problems with multiplicities, i.e., with repeated eigenvalues in the sequence $\Lambda=(\lambda_1,\ldots,\lambda_m)$. For example, the equation $\mathcal D_i^2\mathcal D_jf=0$ has the canonical solution
\[
E_{(i,i,j)}(x)=\sum_{n=2}^{+\infty}\frac{x^{n}}{n!}\sum_{|K|=n-2}i^{k_1+k_2}j^{k_3}.
\]
The sum is the entire function
\begin{align*}
E_{(i,i,j)}(x)&=\frac14\left(x^2\cos x+x\sin x\right)+\frac14\left(-x\cos x+(x^2+1)\sin x \right)i\\
&\quad +\frac14\left(x\cos x+(x^2-1)\sin x\right)j +\frac14\left(-x\cos x+\sin k\right)k.
\end{align*}

(4)\quad In order to illustrate the effect of the non-commutativity of the operators $\DL{\ell}$,  we also compute the canonical solution 
\[
E_{(i,j,i)}(x)=\sum_{n=2}^{+\infty}\frac{x^{n}}{n!}\sum_{|K|=n-2}i^{k_1}j^{k_2}i^{k_3}
\]
of the equation $\mathcal D_i\mathcal D_j\mathcal D_if=0$. We get
\begin{align*}
E_{(i,j,i)}(x)&=\frac14\left(x^2\cos x+x\sin x\right)+\frac14\left(-x\cos x+(x^2+1)\sin x \right)i\\
&\quad +\frac12\left(-x\cos x+\sin x\right)j.
\end{align*}
\end{examples}

\subsection*{The non-homogeneous eigenvalue problem}
We return to the non-homogeneous eigenvalue equation $\Dl f=h$ \eqref{eig:s_nonhom} with a more general right-hand side $h$. Let 
\[h(x)=\sum_{n=0}^{+\infty}x^na_n\in\SR(\hh)\]
be an entire function. As suggested by the definition of the generalized exponential $g^{\mu,\lambda}=E_{(\mu,\lambda)}$ in Proposition \ref{pro:lmII}, we put 
\begin{equation}\label{eq:El}
\El(h):=\sum_{n=1}^{+\infty}\frac{x^n}{n!}\sum_{k=0}^{n-1}k!a_k\lambda^{n-k-1}.
\end{equation}

\begin{proposition}\label{pro:El}
Let $\lambda\in\hh$. If $\lambda\ne0$, assume that there exists a positive real constant $C$ and an integer $d\in\nn$ such that $n!|a_n|\le C (n+1)^d|\lambda|^n$ for every $n\in\nn$. 
%\marginpar{$\Rightarrow$ order of $h\le1$ ?}
Then the series \eqref{eq:El} converges uniformly on the compact sets of $\hh$ to an entire slice-regular function such that $\Dl(\El(h))=h$. As a consequence, $\Dl(g\cdot \El(h))=g\cdot h$ for every slice-constant $g$ on an axially symmetric domain $\OO$ in $\hh$. 
\end{proposition}

\begin{proof}
We prove that the series \eqref{eq:El} is uniformly convergent. The case $\lambda=0$ is immediate from the definition $\mathcal E_0(h)=\sum_{n=1}^{+\infty}x^na_{n-1}/n$. Assume $\lambda\ne0$ and let 
\[e_n(x):=\frac{x^n}{n!}\sum_{k=0}^{n-1}k!a_k\lambda^{n-k-1}.
\]
Let $n\ge1$. From the estimate
\begin{align*}
|e_n(x)|&\le C\frac{|x|^n}{n!}\sum_{k=0}^{n-1}(k+1)^d|\lambda|^{n-1}\le
C\frac{|x|^n}{n!}n^{d+1}|\lambda|^{n-1}= C\frac{|x||2x\lambda|^{n-1}}{(n-1)!}\frac{n^d}{2^{n-1}}\\
&\le C|x|e^{|2x\lambda|}\frac{n^d}{2^{n-1}}
\end{align*}
% \[
% |e_n(x)|\le C\frac{|x|^n}{n!}\sum_{k=0}^{n-1}|\lambda|^{n-1}=
% C\frac{|x|^n}{(n-1)!}|\lambda|^{n-1}= C\frac{|x||2x\lambda|^{n-1}}{(n-1)!}2^{-n+1}
% \le C|x|e^{|2x\lambda|}2^{-n+1}
% \]
we directly observe the uniform convergence of the series \eqref{eq:El}.
It holds
\begin{align*}
\Dl(\El(h))&=\dd{}{x}\left(\sum_{n=1}^{+\infty}\frac{x^n}{n!}\sum_{k=0}^{n-1}k!a_k\lambda^{n-k-1}\right)-\left(\sum_{n=1}^{+\infty}\frac{x^n}{n!}\sum_{k=0}^{n-1}k!a_k\lambda^{n-k-1}\right)\lambda\\
&=\sum_{n=1}^{+\infty}\frac{nx^{n-1}}{n!}\sum_{k=0}^{n-1}k!a_k\lambda^{n-k-1}-\sum_{n=1}^{+\infty}\frac{x^n}{n!}\sum_{k=0}^{n-1}k!a_k\lambda^{n-k}\\
&=\sum_{n=0}^{+\infty}\frac{x^{n}}{n!}\sum_{k=0}^{n}k!a_k\lambda^{n-k}-\sum_{n=1}^{+\infty}\frac{x^n}{n!}\sum_{k=0}^{n-1}k!a_k\lambda^{n-k}\\
&=a_0+\sum_{n=1}^{+\infty}x^n a_n=h(x).
\end{align*}
\end{proof}

For every $\lambda\in\hh\setminus\{0\}$, we introduce an $\hh$-submodule of the right $\hh$-module of entire slice-regular functions on which it is possible to extend the operator $\El$. Let $\Al\subset\SR(\hh)$ be defined as
\[
\Al:=\big\{f=\sum_{n=0}^{+\infty}x^na_n\,|\,
 \text{$\exists C>0$,  $d\in\nn$ such that $n!|a_n|\le C (n+1)^d|\lambda|^n\ \forall n\in\nn$}\big\}.
\]

\begin{proposition}\label{pro:Al}
Let $\lambda\in\hh\setminus\{0\}$. The operators $\Dl$ and $\El$ map $\Al$ into $\Al$. The operator $\El$ is a right inverse of $\Dl$ on $\Al$.
\end{proposition}
\begin{proof}
Let $f=\sum_{n=0}^\infty x^na_n\in\Al$. Then
\[
\Dl f=\sum_{n=1}^\infty x^{n-1}na_n-\sum_{n=0}^\infty x^na_n\lambda=\sum_{n=0}^\infty x^n\left((n+1)a_{n+1}-a_n\lambda\right)=:\sum_{n=0}^\infty x^nb_n.
\]
There exists a positive real $C'$ such that
\[
n!|b_n|=|(n+1)!a_{n+1}-n!a_n\lambda|\le C(n+2)^d|\lambda|^{n+1}+C(n+1)^d|\lambda|^{n+1}\le C'(n+1)^d|\lambda|^n
\]
for every $n\in\nn$, i.e., $\Dl f\in\Al$. Now let 
\[
\El(f)=\sum_{n=1}^{+\infty}\frac{x^n}{n!}\sum_{k=0}^{n-1}k!a_k\lambda^{n-k-1}=:\sum_{n=1}^{+\infty}x^nc_n.
\]
Then 
\[
n!|c_n|=\left|\sum_{k=0}^{n-1}k!a_k\lambda^{n-k-1}\right|\le C\sum_{k=0}^{n-1}(k+1)^d|\lambda|^{n-1}\le C|\lambda|^{-1}(n+1)^{d+1}|\lambda|^n,
\]
and this means that also $\El(f)\in\Al$. The last statement is a consequence of Proposition \ref{pro:El}.
\end{proof}

\begin{example}
Let $h(x)=\sin x$. Then 
\[
\mathcal E_i(h)=\sum_{n=1}^{+\infty}\frac{x^n}{n!}\sum_{h=0}^{\left\lfloor n/2-1\right\rfloor}(-1)^h i^{n-2h-2}=\frac12 x\sin x+\frac12\left(\sin x-x\cos x\right)i
\]
solves $\mathcal D_if=\sin x$, with $\mathcal E_i(h)(0)=0$. Since $h$ is slice-preserving (its series coefficients $a_n$ are all real), the solution $\mathcal E_i(h)$ can also be obtained from the complex solution.
\end{example}

\begin{remark}
The previous example suggests an alternative method to obtain the entire  function $\El(h)$ for a quaternionic function $h\in\Al$. In view of \cite[Lemma 6.1]{GMP}, given a real basis $\{1,\jmath,\kappa,\delta\}$ of $\hh$, every quaternionic slice function $h$ can be written as $h=h_0+h_1\jmath+h_2\kappa+h_3\delta$, with $h_i$ slice-preserving functions for $i=0,1,2,3$. Moreover, $h$ is slice-regular on a domain if and only all the functions $h_i$ are. It is also easy to verify that $h$ belongs to $\Al$ if and only if all the components $h_i\in\Al$.
If $\lambda\ne0$, given any real basis $\{\lambda,\lambda_1,\lambda_2,\lambda_3\}$ of $\hh$, also the set $\{1,\lambda_1\lambda^{-1},\lambda_2\lambda^{-1},\lambda_3\lambda^{-1}\}$ is a real basis of $\hh$. Therefore for any function $h\in\Al$ one can write
\[
h=h_0+h_1(\lambda_1\lambda^{-1})+h_2(\lambda_2\lambda^{-1})+h_3(\lambda_3\lambda^{-1})
\]
with slice-preserving functions $h_i$ in $\Al$. If $\lambda\in\cc_I$, the solution $f_i:=\El(h_i)$ of the equation \eqref{eig:s_nonhom} can be obtained for every $i=0,\ldots,3$ from the complex solution in $\cc_I$. Observe that the functions $f_i$ are one-slice preserving, since $f_i(\cc_I)\subseteq\cc_I$. We then have
\[
\Dl(f_0+\textstyle\sum_{i=1}^3(\lambda_i\lambda^{-1})\cdot f_i)=h_0+\sum_{i=1}^3(\lambda_i\lambda^{-1})\cdot h_i=
h_0+\sum_{i=1}^3 h_i (\lambda_i\lambda^{-1})=h.
\]
Here we point out that the second equality holds since the functions $h_i$ are slice-preserving. 
We then obtain that $f=f_0+\textstyle\sum_{i=1}^3(\lambda_i\lambda^{-1})\cdot f_i$ is the uniquely determined  entire solution $\El(h)$ of \eqref{eig:s_nonhom} that vanishes at the origin.
\end{remark}

\section{Eigenvalue problems for axially monogenic functions}
Let $\dcf$ denotes the Cauchy-Riemann-Fueter operator
\[\dcf =\frac12\left(\dd{}{x_0}+i\dd{}{x_1}+j\dd{}{x_2}+k\dd{}{x_3}\right),\]
and let $\cdcf$ be the conjugated operator
\[\cdcf =\frac12\left(\dd{}{x_0}-i\dd{}{x_1}-j\dd{}{x_2}-k\dd{}{x_3}\right).\]
Let $\OO$ be an axially symmetric domain of $\hh$. 
Our aim is to apply the results of the previous sections to the following eigenvalue problem for the (conjugated) Cauchy-Riemann-Fueter operator $\cdcf$ in the class of \emph{monogenic slice functions} (also called \emph{Fueter-regular functions} in the present quaternionic case):
%\marginpar{Ref. mentioned - see text below}
\begin{equation}\label{eig:am}
\begin{cases}\displaystyle
\cdcf f=f\lambda\text{\quad on $\OO$}\\
 \text{with $f$ slice function s.t.\ }\dcf f=0\text{ on $\OO$ and $\lambda\in\hh$}.
\end{cases}
\end{equation}
Eigenvalue problems of the similar form $\dcf f = \lambda f $ have been studied rather extensively within the very general context where $f$ is some arbitrary function simply belonging to the function space $C^1(\Omega)$ without claiming any further conditions or properties on the considerable $f$ neither on the domain $\Omega$. 

The solutions of $(\dcf - \lambda)f = 0$ where $\dcf$ is the classical quaternionic Cauchy-Riemann-Fueter operator can be described in the form $e^{\lambda x_0} f$ where $f$ is an element of the kernel of $\dcf$, cf. for example \cite{GTD2004}, where the most general context of polynomial Cauchy-Riemann equations of general integer order $n$ and general multiplicity of the eigenvalues has been addressed extensively in the Clifford analysis setting, but under the assumption that $f$ is a general function from $C^n(\Omega)$. This general treatment also includes the special case of $k$-monogenic functions considered more than two decades earlier by F.~Brackx in \cite{Brackx1976}.   

Notice that above we consider slightly differently its conjugated operator $\cdcf$ and the action of $\lambda$ from the right hand-side. 

F.~Sommen and Xu Zhenyuan also studied in \cite{Xu1991,SommenXu1992} the analogous equation in the vector formalism where the Dirac operator is considered instead of the Cauchy-Riemann-Fueter operator exploiting decompositions in terms of axially monogenic functions which have been considered in the preceding work \cite{Sommen1988}. The particular three-dimensional case has already been treated by K.~G\"urlebeck in \cite{G86}. 
The conjugated Dirac operator coincides with the Dirac operator up to a minus sign. Also this operator is a first order operator factorizing the Laplacian. The connection to the Helmholtz operator has been explicitly addressed in \cite{G86}. See also \cite{KS1993} where this topic has been investigated more extensively.   
Even more generally, one also considered the case where $f$ belongs to a general Sobolev spaces, for instance to $H^1_p(\Omega)$, cf. \cite{GS1989}, however again without considering any further (symmetric) conditions. 

In our framework we now address the special situation where $f$ has additionally the property of being a slice function defined over an axially symmetric domain $\Omega$. These analytic and geometric aspects are new. This particular context allows us to use special series representations and the special properties of slice functions which do not hold in the general context in the above mentioned papers, as it will be explained in the sequel.

In order to tackle problem \eqref{eig:am}, we apply Fueter's Theorem, which can be seen as a bridge between the classes of slice-regular functions and monogenic functions.

We recall a useful concept introduced in \cite{GhPe_AIM}. Given a slice function $f$ on $\OO$, the function $f'_s:\OO \setminus \R \to \hh$, called  \emph{spherical derivative} of $f$, is defined as
\[
f'_s(x):=\tfrac{1}{2} \im(x)^{-1}(f(x)-f(\overline x)).
\] 
The function $f'_s$ is a slice function, constant on 2-\emph{spheres}\, $\s_x=\alpha+\s\beta$ for any $x=\alpha+J\beta\in\OO\setminus\R$. For any slice-regular function $f$ on $\OO$, $\sd{f}$ extends as the slice derivative $\dd{f}{x}$ on $\OO\cap\R$.
We recall also a result proved in \cite[Corollary~3.6.2 and Theorem~3.6.3]{Harmonicity} about some formulas linking the spherical derivative of slice functions with the Cauchy-Riemann-Fueter operator.

\begin{theorem}[\cite{Harmonicity}]\label{thm:harmonicity}
Let $\OO$ be an axially symmetric domain in $\HH$. Let  $f:\OO\to\HH$ be a slice function of class $\mathcal{C}^1(\OO)$. Then
\begin{enumerate}
\item
$f$ is slice-regular if and only if $\,\dcf f=-f'_s$.

\item
If $f:\OO\to\HH$ is slice-regular, then it holds:
\begin{enumerate}
\item
  The four real components of $f'_s$ are harmonic on $\OO$. %$\Delta f'_s=0$ on $\OO$. 
\item
The following generalization of Fueter's Theorem holds:
\[\dcf\Delta f=\Delta \dcf f=-\Delta f'_s=0.\]
\item
$\Delta f=-4\,\dd{\sd f}{x}$.
\end{enumerate}
\end{enumerate}
\end{theorem}

In the following we will use also the following result.

\begin{lemma}\label{lem:affine}
(a)\quad 
If $f\in\SR(\OO)$ and $\Delta f=0$, then $f$ is a quaternionic affine function of the form $xa+b$, with $a,b\in\hh$.\\
(b)\quad The Laplacian of a slice-regular function can be expressed by first order derivatives. If $f\in\SR(\OO)$, for every $x\in\OO\setminus\rr$ it holds
\begin{align}\label{eq:Deltaf1}
\Delta f&=-2\frac{\IM(x)}{|\IM(x)|^2}\left(\sd f-\dd{f}{x}\right)=2\frac{\IM(x)}{|\IM(x)|^2}\left(\dcf f+\dd{f}{x}\right)\\
&=\frac{\IM(x)}{|\IM(x)|^2}\left(3\dd{f}{x_0}+i\dd{f}{x_1}+j\dd{f}{x_2}+k\dd{f}{x_3}\right).\label{eq:Deltaf2}
\end{align}
\end{lemma}
\begin{proof}
Let $f=\I(F_1+iF_2)$, $z=\alpha+i\beta$.

(a)\quad From Theorem \ref{thm:harmonicity}(2c), it follows that $\dd{\sd f}{x}=0$, i.e., $\sd f=\I(\beta^{-1}F_2(\alpha+i\beta))$ is anti-slice-regular on $\OO$. Then the function $\beta^{-1}F_2(\alpha+i\beta)$ is locally constant, i.e., $F_2=\beta a$ with $a\in\hh$. Since $F$ is holomorphic, it follows that $F(z)=\alpha a+b+i\beta a=za+b$, and $f=\I(F)=xa+b$, with $b\in\hh$.

(b)\quad $\Delta f=-4\dd{\sd f}{x}=-4\I\left(\dd{}{z}\left(\frac{F_2}{\beta}\right)\right)$ and 
\[
\dd{}{z}\left(\frac{F_2}{\beta}\right)=\frac12(\partial_\alpha-i\partial_\beta)\left(\frac{F_2}{\beta}\right)=
\frac1{2\beta}(\partial_\alpha-i\partial_\beta)F_2+\frac{i}{2\beta^2}F_2=-\frac{i}{2\beta}\dd{F}{z}+\frac{i}{2\beta^2}F_2,
\]
since $F$ is holomorphic. Therefore, 
\[
\dd{\sd f}{x}=\I\left(\frac{i}{2\beta}\left(\frac{F_2}{\beta}-\dd{F}{z}\right)\right)=
\frac{\IM(x)}{2|\IM(x)|^2}\left(\sd f-\dd{f}{x}\right)
\]
and using Theorem \ref{thm:harmonicity}(1) we get formula \eqref{eq:Deltaf1}. Since $f$ is slice-regular, we have $\dd{f}{x}=\dd{f}{x_0}$ (see e.g.\ \cite{GeSt2007Adv}). This equality and \eqref{eq:Deltaf1} give formula \eqref{eq:Deltaf2}.
\end{proof}

Let $\AM(\OO)$ denote the class of axially monogenic functions, i.e., of monogenic slice functions on $\OO$:
\[\AM(\OO)=\{f\in\SF^1(\OO)\,|\,\dcf f=0\}.
\]
In view of the generalized Fueter's Theorem, the Laplacian maps the space $\SR(\OO)$ into $\AM(\OO)$. 
It is known that this map is surjective (see e.g.\ \cite{TheInverseFueter}).
%\marginpar{check?}
Now we construct an operator $\LL:\AM(\OO)\to\AM(\OO)$ which makes the following diagram
\[
\begin{CD}
\SR(\OO) @>\dd{}{x}> >\SR(\OO)\\ % right arrow with labels
@V \Delta V   V %down arrow with labels
@V V \Delta V\\%down arrow with labels
\AM(\OO) @>\LL> >\AM(\OO) % right arrow with labels
\end{CD} 
\]
commutative. Chosen a right inverse $\ID:\AM(\OO)\to\SR(\OO)$ of $\Delta$, we can set $\LL:=\Delta\circ\dd{}x\circ\ID$. To define $\ID$, we start from polynomials. For every $n\in\nn$, let
\begin{equation}\label{eq:defPn}
\Pn_n(x):=-\frac14\Delta(x^{n+2})=
\dd{}{x}\sd{(x^{n+2})}=\dd{\Zt_{n+1}(x)}{x}.
\end{equation}
Here $\Zt_{n}(x):=\sd{(x^{n+1})}$ is a harmonic homogeneous polynomial of degree $n-1$ in the four real variables $x_0$, $x_1$, $x_2$, $x_3$. The polynomials $\Zt_{n}$ are called \emph{zonal harmonic polynomials with pole 1}, since they have an axial symmetry with respect to the real axis (see \cite{Harmonicity,AlmansiH}). Observe that the polynomials $\Pn_n$ are axially monogenic but not zonal.

According to our knowledge, zonal monogenics were introduced in \cite{SommenJ1997}. See also \cite{Eelbode} where further basic properties have been studied. 

\begin{proposition}
The polynomials $\Pn_n$ are axially monogenic homogeneous polynomial of degree $n$ in $x_0$, $x_1$, $x_2$, $x_3$. They are slice functions on $\hh$, given by the explicit formula
%\marginpar{$\Pn_n$ are inner spherical monogenics of order $n$ (outer for $n<0$)}
\begin{equation}\label{eq:pn}
\Pn_n(x)=\sum_{k=1}^{n+1}kx^{k-1}\overline x^{n-k+1}=\I(\sum_{k=1}^{n+1}kz^{k-1}\overline z^{n-k+1}).
\end{equation}
In particular, $\Pn_n$ is slice-preserving and the restriction of $\Pn_n$ to the real axis is the monomial $\Pn_n(x_0)=\binom{n+2}2x_0^n$.
The inducing stem functions $p_n(z):=\sum_{k=1}^{n+1}kz^{k-1}\overline z^{n-k+1}$ satisfy the Vekua equation
\begin{equation}\label{eq:Vekua}
\dd{f}{\overline z}=\frac{f-\overline f}{z-\overline z}\text{\quad for $z\in\cc\setminus\rr$}.
\end{equation}
\end{proposition}
\begin{proof}
The first statement is a consequence of Fueter's Theorem.
Formula \eqref{eq:pn} follows from \cite[Corollary 6.7]{Harmonicity}, since
\[
\Pn_n(x)=\dd{}{x}\sd{(x^{n+2})}=\dd{}{x}(\sum_{k=0}^{n+1}x^{k}\overline x^{n-k+1})=\sum_{k=1}^{n+1}kx^{k-1}\overline x^{n-k+1}.
\]
To prove \eqref{eq:Vekua}, we observe that
\begin{align*}
(z-\overline z)\dd{p_n}{\overline z}&=(z-\overline z)\sum_{k=1}^{n}k(n-k+1)z^{k-1}\overline z^{n-k}\\
&=\sum_{k=1}^{n}k(n-k+1)z^{k}\overline z^{n-k}-\sum_{k=1}^{n}k(n-k+1)z^{k-1}\overline z^{n-k+1}\\
&=\sum_{k=1}^{n}k(n-k+1)z^{k}\overline z^{n-k}-\sum_{k=0}^{n-1}(k+1)(n-k)z^{k}\overline z^{n-k}=\sum_{k=0}^{n}(2k-n)z^{k}\overline z^{n-k},
\end{align*}
while
\begin{align*}
p_n(z)-\overline{p_n(z)}&=\sum_{k=1}^{n+1}kz^{k-1}\overline z^{n-k+1}-\sum_{k=1}^{n+1}k\overline z^{k-1} z^{n-k+1}\\
&=\sum_{k=1}^{n+1}kz^{k-1}\overline z^{n-k+1}-\sum_{k=1}^{n+1}(n-k+2) z^{k-1}\overline z^{n-k+1}\\
&=\sum_{k=1}^{n+1}(2k-n-2)z^{k-1}\overline z^{n-k+1}=\sum_{k=0}^{n}(2k-n)z^{k}\overline z^{n-k}.
\end{align*}
\end{proof}

\begin{example}
The first four axially monogenic polynomials $\Pn_n$ are 
\[\begin{cases}
\Pn_0(x)=1,\\
 \Pn_1(x)= 3 x_0+x_1i+x_2j+x_3k,\\
  \Pn_2(x)=6 x_0^2 - 2x_1^2 - 2x_2^2 - 2x_3^2+4x_0(x_1i+x_2j+x_3k),\\
  \Pn_3(x)=10x_0(x_0^2-x_1^2-x_2^2-x_3^2)+2(5x_0^2-x_1^2-x_2^2-x_3^2)(x_1i+x_2j+x_3k).
\end{cases}
\]
\end{example}

We define $\ID:\AM(\hh)\cap\hh[x_0,x_1,x_2,x_3]\to\hh[x]$ by extending linearly the mapping that associates the monomial $-\frac14x^{n+2}$ to the polynomial $\Pn_n$
\[
\ID\left(\textstyle\sum_{n=0}^d\Pn_n(x)a_n\right):=-\tfrac14\textstyle\sum_{n=0}^d x^{n+2}a_n.
\]
By applying the definitions, it follows that $\Delta\ID\Pn_n=\Pn_n$ for every $n\in\nn$, while $\ID\Delta x^n=x^n$ for every integer $n\ge2$, and of course $\ID\Delta x^n=0$ for $n=0,1$.

The mapping $\ID$ can be extended linearly to convergent series $\sum_{n\in\nn}\Pn_n(x)a_n\in\AM(\OO)$. 
If $f(x)=\sum_{n\in\nn}\Pn_n(x)a_n$, then $f(x_0)=\sum_{n\in\nn}\binom{n+2}2x_0^na_n$ for real $x_0$. Therefore 
\[a_n=\frac2{(n+2)!}\frac{\partial^n f}{\partial x_0^n}(0)\text{\quad for $n\in\nn$}
\]
and we can write
\[
\ID(f)=-\frac12\sum_{n\in\nn} \frac{x^{n+2}}{(n+2)!}\frac{\partial^n f}{\partial x_0^n}(0).
\]

\begin{proposition}\label{pro:L}
Let $\LL=\Delta\circ\dd{}x\circ\ID:\AM(\OO)\to\AM(\OO)$. 
Then it holds $\LL(\Pn_n)=(n+2)\Pn_{n-1}$ and $\cdcf\Pn_n=(n+2)\Pn_{n-1}$ for every $n\in\nn,n\ge1$. As a consequence, $\LL$ coincides with the conjugated Cauchy-Riemann-Fueter operator $\cdcf$ on $\AM(\OO)$.
%, i.e.,  $\cdcf=\Delta\circ\dd{}x\circ\ID$.
\end{proposition}
\begin{proof}
By direct computation, using Theorem \ref{thm:harmonicity}(2c) and Definition \eqref{eq:defPn}, we get
\[
\Delta\left(\dd{(\ID\Pn_n)}{x}\right)=-\frac14\Delta\left(\dd{x^{n+2}}{x}\right)=
-\frac14\Delta\left((n+2)x^{n+1}\right)=(n+2)\dd{\sd{(x^{n+1})}}{x},
\]
that is $\LL(\Pn_n)=(n+2)\Pn_{n-1}$. On the other hand, since $\Pn_n$ is in the kernel of $\dcf$, it holds 
\[\cdcf\Pn_n=(\cdcf+\dcf)\Pn_n=\dd{\Pn_n}{x_0}=-\frac14\dd{(\Delta x^{n+2})}{x_0}=-\frac14\Delta\left(\dd{x^{n+2}}{x_0}\right).
\]
Since $x^{n+2}$ is slice-regular, we have $\dd{x^{n+2}}{x_0}=\dd{x^{n+2}}{x}$ (see e.g.\ \cite{GeSt2007Adv}) and then as before
\[
\cdcf\Pn_n=-\frac14\Delta\left(\dd{x^{n+2}}{x}\right)=(n+2)\Pn_{n-1}
\]
for every integer $n\ge1$.
\end{proof}

\begin{remark}
Definition \eqref{eq:defPn} can be extended to negative indices $n\in\zz$. One obtains axially monogenic functions $\Pn_n\in\AM(\hh\setminus\{0\})$ that satisfy the same relation as in the case $n\ge0$: $\cdcf\Pn_n=(n+2)\Pn_{n-1}$ for every $n<0$. Observe that $\Pn_{-1}=\Pn_{-2}=0$. The functions $\Pn_n$ and  $\Pn_{-n}$ are related through the Kelvin transform of $\rr^4$ (see \cite[Prop.5.1(c)]{Harmonicity}). It follows that also for negative $n$ the $\Pn_n$'s are homogeneous of degree $n$.
\end{remark}

Proposition \ref{pro:L} implies that the following diagram
\[
\begin{CD}
\SR(\OO) @>\dd{}{x}> >\SR(\OO)\\ % right arrow with labels
@V \Delta V   V %down arrow with labels
@V V \Delta V\\%down arrow with labels
\AM(\OO) @>\cdcf> >\AM(\OO) % right arrow with labels
\end{CD} 
\]
is commutative, since it holds $\LL\circ\Delta(x^n)=\cdcf\circ\Delta(x^n)=\Delta\circ\dd{x^n}{x}$ for every $n\in\nn$.

Let $\lambda\in\hh$ and let $\LL_\lambda=\LL-R_\lambda=\cdcf-R_\lambda$ denote the linear operator mapping a  function $f\in\AM(\OO)$  to the monogenic slice function 
\[\LL_\lambda f=\cdcf f -f\lambda.\]
The solutions of the eigenvalue problem \eqref{eig:am} are exactly the elements of the kernel of $\LL_\lambda$. Notice that $\LL_\lambda=\Delta\circ\Dl\circ\ID$ and that $\LL_\lambda\circ\Delta=\Delta\circ\Dl$ for every $\lambda\in\hh$. 
From this relation and Proposition \ref{pro:sol_hom} we can deduce the following characterization of $\ker(\LL_\lambda)$.

\begin{proposition}\label{pro:am}
The axially monogenic function $\Delta\exp_\lambda(x)\in\AM(\hh)$ is a solution  of \eqref{eig:am} on $\hh$. If $\lambda\ne0$ and $\OO$ is a slice domain, a function $f\in\AM(\OO)$ is a solution of \eqref{eig:am} if and only if $f=c\cdot\Delta\exp_\lambda(x)$, with $c\in\hh$. If $\lambda=0$, $\ker(\LL)$ contains only the constants.
If $\OO$ is a product domain and $\lambda\ne0$, the operator $\LL_\lambda$ on $\AM(\OO)$ has kernel
\[
\ker(\LL_\lambda)=\{\Delta(g\cdot\exp_\lambda(x))\;|\; g\in\SC(\OO)\}.
\]
If $\lambda=0$,  $\ker(\LL)=\{c+\Delta g\;|\;c\in\hh, g\in\SC(\OO)\}$.
The function $\Delta\exp_\lambda(x)$ has the following expansion 
\[
\Delta\exp_\lambda(x)=-4\sum_{n=0}^{+\infty}\Pn_n(x)\frac{\lambda^{n+2}}{(n+2)!}=-4\sum_{n=0}^{+\infty}\sum_{k=1}^{n+1}kx^{k-1}\overline x^{n-k+1}\frac{\lambda^{n+2}}{(n+2)!}.
\]
\end{proposition}
\begin{proof}
Let $\lambda\ne0$. 
The equality $\Dl f=0$ implies that $\LL_\lambda(\Delta f)=0$. Conversely, if $\LL_\lambda(\Delta f)=\Delta(\Dl f)=0$, with $f\in\SR(\OO)$, then the function $\Dl f$ is a quaternionic affine function of the form $xa+b$, with $a,b\in\hh$ (Lemma \ref{lem:affine}).
From Proposition \ref{pro:polynomial} it follows that $f=-a\lambda^{-2}-b\lambda^{-1}-x(a\lambda^{-1})$, up to an element of $\ker(\Dl)$. Therefore $\Delta f=\Delta g$, with $g\in\ker(\Dl)$. If $\OO$ is a slice domain, then $\Delta f=\Delta(c\cdot\exp_\lambda(x))=c\cdot\Delta\exp_\lambda(x)$, with $c\in\hh$. If $\OO$ is a product domain, then $\Delta f=\Delta(g\cdot\exp_\lambda(x))$, with $g\in\SC(\OO)$.

If $\lambda=0$ and $\LL(\Delta f)=0$, then $\Delta(\dd{f}{x})=0$ and therefore $\dd{f}{x}=xa+b$ is affine. Since $f\in\SR(\OO)$, it follows that $f=x^2 a/2+bx+c+g$, with $a,b,c\in\hh$ and $g\in\SR(\OO)\cap\ker(\dd{}{x})=\SC(\OO)$. 
If $\OO$ is a slice domain, we get that $\Delta f$ is a constant. If $\OO$ is a product domain, $\Delta f=-2a+\Delta g$, with $g\in\SC(\OO)$ and then $\ker(\LL)=\{c+\Delta g\;|\;c\in\hh, g\in\SC(\OO)\}$.
\end{proof}

\begin{example}\label{ex:expj}
Since $\exp_j(x)=\cos x+(\sin x)j$, from Lemma \ref{lem:affine} it follows that
\[
\Delta \exp_j(x)%&=-2\tfrac{\IM(x)}{|\IM(x)|^2}\left(\sd{(\exp_j(x))}-\exp_j(x)j\right)\\
=-2\tfrac{\IM(x)}{|\IM(x)|^2}\left(\sd{(\cos x)}+\sd{(\sin x)}j+\sin x-(\cos x) j\right).
\]
A direct computation shows that
\[
\sd{(\cos x)}=\frac{e^{-\beta}-e^\beta}{2\beta}\sin(x_0),\quad \sd{(\sin x)}=\frac{e^{\beta}-e^{-\beta}}{2\beta}\cos(x_0),
\]
where $\beta=\sqrt{x_1^2+x_2^2+x_3^2}$. The axially monogenic function $\Delta\exp_j(x)$ on $\hh$ satisfies $\LL_j(\Delta\exp_j(x))=0$, i.e., $\partial (\Delta\exp_j(x))=\Delta\exp_j(x)j$. 
\end{example}

\begin{remark}\label{rem:mui}
If $\OO=\hh\setminus\rr$, a product domain, the results of Proposition \ref{pro:am} can be made more explicit.
Let $I\in\s$ and $\mu_I\in\SC(\hh\setminus\rr)$ as in \eqref{eq:muI}. The equality $\Delta\mu_I=-\frac{\IM(x)}{|\IM(x)|^3}I$ and formula \eqref{eq:muJK} imply that 
\[\ker(\LL)=\ker(\cdcf)\cap\ker(\dcf)=\left\{a+\tfrac{\IM(x)}{|\IM(x)|^3}b\;|\;a,b\in\hh\right\},\]
while for $\lambda\ne0$ the elements of $\ker(\LL_\lambda)$ are $\hh$-linear combinations of functions of the form
\[
\Delta(\mu_I\cdot\exp_\lambda(x))=\sum_{n=0}^{+\infty}\Delta(x^n\mu_I(x))\frac{\lambda^n}{n!}.
\]
\end{remark}

Now consider the non-homogeneous equation \eqref{eig:s_nonhom}, with $f,h\in\SR(\OO)$. The equality $\Dl f=h$ implies that $\LL_\lambda(\Delta f)=\Delta h$. Conversely, if $\LL_\lambda(\Delta f)=\Delta h$ for $f,h\in\SR(\OO)$, then $\Delta(\Dl f-h)=0$ and thanks to Lemma \ref{lem:affine} the function $\Dl f-h$ is a quaternionic affine function of the form $xa+b$, with $a,b\in\hh$. If $h\in\hh[x]$ and $\lambda\ne0$, then Proposition \ref{pro:polynomial} gives $f=\Sl(h+xa+b)+g=\Sl(h)+xa'+b'+g$, with $a',b'\in\hh$ and $g\in\ker(\Dl)$. Here $\Sl$ is the solution operator defined in \eqref{def:Sl}. 
Therefore $\Delta f=\Delta\Sl(h)+\Delta g$, with $g\in\ker(\Dl)$.

\subsection*{Eigenvalue problem of the $m$th-order for axially monogenic functions}
Now we generalize Proposition \ref{pro:am} to eigenvalue problems of any order $m$. Let $\lambda_1,\ldots,\lambda_m\in\hh$ and consider the equation
\begin{equation}\label{eq:am_eig_m}
\begin{cases}\displaystyle
\LLL{1}\cdots\LLL{m} f=0\text{\quad on $\OO$},\\
\text{with }f\in\AM(\OO).
\end{cases}
\end{equation}

\begin{definition}\label{def:DeltaE}
The \emph{generalized $\Delta$-exponential function} associated with the $n$-tuple $\Lambda=(\lambda_1,\ldots,\lambda_m)\in\hh^n$ is the axially monogenic function $E^\Delta_\Lambda(x)$ defined on $\hh$ by
\[
E^\Delta_\Lambda(x):=\Delta(E_\Lambda(x))=-4\sum_{n=m-1}^{+\infty}\frac{\Pn_{n-2}(x)}{n!}\sum_{|K|=n-m+1}\lambda_1^{k_1}\cdots\lambda_m^{k_m},
\]
where the sum is extended over the multi-indices $K=(k_1,\ldots,k_m)$ of non-negative integers such that $|K|=k_1+\cdots+k_m=n-m+1$. In particular, it holds $E^\Delta_{(\lambda)}=\Delta\exp_{\lambda}(x)$ for every $\lambda\in\hh$.
\end{definition}

In view of formula \eqref{eq:pn}, we have 
\[
E^\Delta_\Lambda(x)=-4\sum_{n=m-1}^{+\infty}\frac{1}{n!}\sum_{k=1}^{n-1}kx^{k-1}\overline x^{n-k-1}
\sum_{|K|=n-m+1}\lambda_1^{k_1}\cdots\lambda_m^{k_m}.
\]
Observe that for real values $\lambda$, the $\Delta$-exponential $E^\Delta_{(\lambda)}(x)$ coincides up to a multiplicative constant with the function $\text{EXP}_3(\lambda x)$ defined in \cite[Ex.11.34]{GHS} in the context of Clifford algebras: \[
E^\Delta_{(\lambda)}(x)=-2e^{\lambda x_0}(\sinc(\lambda\beta)-I_x((\sinc(\lambda\beta))')
\]
where $I_x=\tfrac{\IM(x)}{|\IM(x)|}$, $\beta=|\IM(x)|$, $\lambda\in\rr$.

\begin{proposition}\label{pro:am_m}
The axially monogenic function $E^\Delta_\Lambda\in\AM(\hh)$ is a solution  of \eqref{eq:am_eig_m}. 
In general, every function of the form
\begin{equation}\label{eq:sol_delta_m}
\sum_{i=1}^m \Delta(h_i\cdot E_{(\lambda_i,\ldots,\lambda_m)}),
\end{equation}
with $h_i\in\SC(\OO)$, is a solution of \eqref{eq:am_eig_m}.
If $\lambda_i\ne0$ for every $i=1,\ldots,m$, all the solutions of \eqref{eq:am_eig_m} are of the form \eqref{eq:sol_delta_m}.
\end{proposition}
\begin{proof}
The first two statements follow immediately from Corollary \ref{cor:order_m}:
\[
\LLL{1}\cdots\LLL{m}(E^\Delta_\Lambda)=\LLL{1}\cdots\LLL{m}(\Delta(E_\Lambda))=\Delta(\DL{1}\cdots\DL{m}E_\Lambda)=0.
\]
The same computation holds for every function of the form \eqref{eq:sol_delta_m}.

Assume that $\lambda_i\ne0$ for every $i$ and that $\LLL{1}\cdots\LLL{m}(\Delta g)=0$. Then it holds $\Delta(\DL{1}\cdots\DL{m}g)=0$, which implies by Lemma \ref{lem:affine} that $\DL{1}\cdots\DL{m}g$ is a quaternionic affine function of the form $xa+b$, with $a,b\in\hh$. Therefore we have
\[
\DL{1}\cdots\DL{m}g=xa+b=\DL{1}\cdots\DL{m}\SF_{\lambda_m}\cdots\SF_{\lambda_1}(xa+b),
\]
where $\Sl$ is the right inverse of $\Dl$ on polynomials defined for any $\lambda\ne0$ in \eqref{def:Sl}. By definition, also $\SF_{\lambda_m}\cdots\SF_{\lambda_1}(xa+b)$ is an affine polynomial $xa'+b'$. Then 
$g-xa'-b'\in\ker(\DL{1}\cdots\DL{m})$ and thanks to Corollary \ref{cor:order_m}, $f=\Delta g$ is of the form \eqref{eq:sol_delta_m}.
\end{proof}

\begin{example}\label{eq33}
Let $\lambda_1=i$, $\lambda_2=j$. The solution of $\mathcal D_i\mathcal D_jg=0$ on $\hh$ given in Examples \ref{ex1} is the slice-regular entire function
\[
E_{(i,j)}(x)=\frac12\left(\sin x+x\cos x+(x\sin x)(i+j)+(\sin x-x\cos x)k\right).
\]
Therefore $f:=\Delta E_{(i,j)}(x)=E^\Delta_{(i,j)}(x)$ is an axially monogenic solution of the equation $\LL_i\LL_j f=0$, that is
\[
(\cdcf-R_i)(\cdcf-R_j)f = \cdcf^2 f-(\cdcf f)(i+j)-fk=0.
\]

\end{example}

\section{Applications}
As a very interesting bi-product we can relate the solutions to $\LLL{1}\LLL{2} f = 0$ to axially monogenic solutions to the 3D and 4D time-harmonic Helmholtz and stationary massless Klein-Gordon equation considered in an axially symmetric domain. In the sequel we abbreviate for convenience the purely quaternionic part of the Cauchy-Riemann-Fueter operator by ${\mathcal{D}}_x := i \frac{\partial }{\partial x_1} + j \frac{\partial }{\partial x_2} + k  \frac{\partial }{\partial x_3}$. Actually, ${\mathcal{D}}_x$ is often called the Euclidean Dirac operator and it satisfies ${\mathcal{D}}_x^2 = - \Delta_3$, where $\Delta_3$ is the ordinary Euclidean Laplacian in $\mathbb{R}^3=\IM(\hh)$. Notice that $\overline{\partial} \partial = \frac{1}{4} \Delta_4$, where $\Delta_4$ now represents the Laplacian in $\mathbb{R}^4$, cf.\ for example \cite{GS1989} and many other classical textbooks on quaternionic and Clifford analysis, such as also \cite{bds} and others. In this section we use the symbol $\Delta_4$ instead of $\Delta$ as in the previous section to avoid confusion with $\Delta_3$.  
As a direct consequence of Proposition \ref{pro:am_m} %the previous relation (Example~\ref{eq33}) 
we can establish 

\begin{proposition}\label{pro:KGH}
Suppose that $f$ is an axially monogenic solution to $\LLL{1}\LLL{2} f = 0$ on some axially symmetric domain $\Omega \subset \mathbb{H}$. 
\begin{itemize}
\item[(a)] Let $I\in\s$ be any imaginary unit. In the particular case where we take $\lambda_1 = I \lambda$ and $\lambda_2 = - I \lambda$ where $\lambda$ is a non-zero real number, the solutions to  $\LL_{I\lambda}\LL_{-I\lambda} f = 0$ are axially monogenic solutions to the massless stationary Klein-Gordon equation $(\Delta_3 - \lambda^2) f = 0$ on $\Omega^{*} := \Omega \cap \mathbb{R}^3$. 
For any finite subset $A$ of $\s$, the functions of the form
\[
f(x)=\sum_{I\in A}\Delta_4(h_I\cdot\exp_{I\lambda}(x)),
\]
where $h_I\in\SC(\OO)$, are axially monogenic solutions of the Klein-Gordon equation.
\item[(b)] In the other particular case we take $\lambda_1 = \lambda$ and $\lambda_2 =- \lambda$ where again $\lambda$ is supposed to be a non-zero real value, the solutions to  $\LL_{\lambda}\LL_{-\lambda} f = 0$  are axially monogenic solutions to the time-harmonic Helmholtz equation $(\Delta_3 + \lambda^2) f = 0$ on $\Omega^{*} := \Omega \cap \mathbb{R}^3$.
For every $h_1,h_2\in\SC(\OO)$, the function
\[
f(x)=\Delta_4(h_1\cdot\exp_{\lambda}(x))+\Delta_4(h_2\cdot\exp_{-\lambda}(x))
\]
is an axially monogenic solution of the Helmholtz equation.
\end{itemize}
\end{proposition}
\begin{proof}
Since $f$ is monogenic and consequently also an element in the kernel of $\Delta_4$, it holds 
\begin{equation*}\label{eq:monogenic}
\mathcal{D}_xf=-\partial_{x_0}f\quad\text{and}\quad \partial^2_{x_0}f=-\Delta_3f.
\end{equation*}
It follows that
\begin{align*}
\partial^2f&=\frac14(\partial_{x_0}-\mathcal D_x)(\partial_{x_0}-\mathcal D_x)f=\frac14(\partial^2_{x_0}f-\partial_{x_0}\mathcal D_xf-\mathcal D_x\partial_{x_0}f+\mathcal D_x^2 f)\\
&=\frac14(3\partial^2_{x_0}f-\Delta_3f)=-\Delta_3f,
\end{align*}
and then
\begin{align*}
\LL_{\lambda_1}\LL_{-\lambda_1}f&=(\partial-R_{\lambda_1})(\partial-R_{-\lambda_1})f\\
&=\partial^2f - \underbrace{\partial f( \lambda_1 - \lambda_1)}_{=0} - f\lambda_1^2\\
&=-\Delta_3f-f\lambda_1^2.
\end{align*}
We first treat the case announced in statement (a): Putting $\lambda_{1,2} = \pm I \lambda$, with $I\in\s$, one obtains that 
\[
0=\LL_{I\lambda}\LL_{-I\lambda}f=-\Delta_3f+\lambda^2 f.
\]
Analogously one obtains in the other case the statement (b) by inserting $\lambda_{1,2} = \pm \lambda$, namely 
\[
0=\LL_{\lambda}\LL_{-\lambda}f=-\Delta_3f-\lambda^2 f.
\]
Since we are in the case of commuting eigenvalues, in view of Proposition \ref{pro:am_m} and Remark \ref{rem:m}(3), the general axially monogenic solution of the equation $\LL_{\lambda_1}\LL_{-\lambda_1}f=0$ is of the form
\[
g(x)=\Delta_4(h_1\cdot\exp_{\lambda_1}(x))+\Delta_4(h_2\cdot\exp_{-\lambda_1}(x)),
\]
with $h_1$ and $h_2\in\SC(\OO)$. This implies the last statements of items (a) and (b).
\end{proof}

\begin{example}
The axially monogenic function $f(x)=\Delta_4\exp_1(x)=\Delta_4e^x$, restricted to $\IM(\hh)=\rr^3$, satisfies the equation $\Delta_3f+f=0$ on $\rr^3$, while the function $g(x)=\Delta_4\exp_j(x)$ of Example \ref{ex:expj} satisfies the equation $\Delta_3g-g=0$. 

Another solution of the Helmholtz equation $\Delta_3f+f=0$ on $\rr^3\setminus\{0\}=\rr^3\cap(\hh\setminus\rr)$ is given by the axially monogenic function (see Remark \ref{rem:mui}) 
\[
f(x)=\Delta_4(\mu_i\cdot\exp_1(x))=\sum_{n=0}^{+\infty}\frac{\Delta_4(x^n\mu_i(x))}{n!},
\]
while the axially monogenic function 
\[
g(x)=\Delta_4(\mu_i\cdot\exp_j(x))=\sum_{n=0}^{+\infty}\Delta_4(x^n\mu_i(x))\frac{j^n}{n!}
\]
is a solution of the Klein-Gordon equation $\Delta_3g-g=0$ on $\rr^3\setminus\{0\}$.
\end{example}

\begin{remark}
The two cases considered in Proposition \ref{pro:KGH} are included in the more general case $\lambda_1=-\lambda_2=\alpha+\beta I\in\hh$, with $\alpha,\beta$ real and $I\in\s$. If $f\in\AM(\OO)$, it holds $\LL_{\lambda_1}\LL_{-\lambda_1}f=0$ if and only if 
\[
\Delta_3f+(\alpha^2-\beta^2)f+(2\alpha\beta)fI=0.
\]
\end{remark}

\begin{remark}\label{rem:Yukawa}
The Klein-Gordon equation is the homogeneous equation associated to the Yukawa equation $\Delta_3 f-\lambda^2 f=h$, with $\lambda$ real.  If $g$ is slice-regular on $\OO$, $h:=-\Delta_4g$ and $I\in\s$, then $h\in\AM(\OO)$ and every slice-regular solution $\tilde f$ of the equation $\mathcal D_{I\lambda}\mathcal D_{-I\lambda}\tilde f=g$ gives a solution $f:=\Delta_4\tilde f$ of the equation 
$\LL_{I\lambda}\LL_{-I\lambda} f=\Delta_4 g$, which is the Yukawa equation $\Delta_3f-\lambda^2 f=h$.

For example, if $h\in\hh[x_0,x_1,x_2,x_3]$ is an axially monogenic polynomial and $g:=-\widetilde\Delta h\in\hh[x]$, where $\widetilde\Delta$ is the right inverse of $\Delta_4$ defined in the previous section, then $\tilde f$ can be obtained by means of the right inverses operators $\mathcal S_{\pm I\lambda}$ of $\mathcal D_{\pm I\lambda}$ introduced in \eqref{def:Sl}: $\tilde f=\mathcal S_{-I\lambda}\mathcal S_{I\lambda}(g)$.
We then get that 
\[f:=-\Delta_4\mathcal S_{-I\lambda}\mathcal S_{I\lambda}(\widetilde\Delta h)
\]
is an axially monogenic solution of the Yukawa equation with right-hand $h$. 
\end{remark}

\begin{example}
Let $h(x)=\Pn_3(x)-\Pn_2(x)(i+j+k)+\Pn_1(x)(i-j+k)+1$. Then $h(x)=-\frac14\Delta_4(x^2(x-i)\cdot(x-j)\cdot(x-k))$. We take $\lambda=1$, $I=i$. A direct computation shows that the solution
\[
f=\tfrac14\Delta_4\mathcal S_{-i\lambda}\mathcal S_{i\lambda}(x^2(x-i)\cdot(x-j)\cdot(x-k))
\]
of the Yukawa equation $\Delta_3 f-f=h$ has the form
\[
f(x)=-\Pn_3(x)+\Pn_2(x)(i+j+k)+\Pn_1(x)(20-i+j-k)-1-12(i+j+k).
\]
\end{example}

% \bibliographystyle{abbrv}
% \bibliography{Ref_AP}

\end{document}